\documentclass[a4paper,reqno]{amsart}
\usepackage[margin=1.5in]{geometry}
\usepackage[english]{babel}
\usepackage[T1]{fontenc}
\usepackage[utf8]{inputenc}
\usepackage{microtype}
\usepackage{stmaryrd}
\usepackage[dvipsnames]{xcolor}
\usepackage{subfiles}

\usepackage{eucal}
\usepackage[shortlabels]{enumitem}
\usepackage[super]{nth}

\usepackage{amsmath}
\usepackage{amssymb}

\usepackage{amsthm}
\usepackage{amsfonts}
\usepackage{graphicx}

\usepackage{dsfont}
\usepackage{mathtools}
\usepackage{xparse}
\usepackage{tikz-cd}
\usepackage{quiver}
\usepackage{nccmath}

\usepackage{tikz}
\usepackage{tabularx}
\newcolumntype{Y}{>{\centering\arraybackslash}X}
\newcommand{\nwidth}{4cm}
\tikzstyle{no} = [rectangle, rounded corners, 
minimum width=\nwidth, 
minimum height=1cm,
text width = \nwidth,
align=center,
draw=black]

\tikzstyle{arrowv} = [implies-,double equal sign distance,shorten >=0.15cm,shorten <=0.15cm]
\tikzstyle{arrow} = [implies-,double equal sign distance,shorten >=0.05cm,shorten <=0.05cm]

\usepackage{xcolor}
\usepackage[pdfborder={0 0 0}]{hyperref}
\hypersetup{
	colorlinks=true,
	allcolors=.,
	bookmarksdepth = 3
}
\usepackage[capitalize,nameinlink,nosort]{cleveref}
\crefformat{equation}{#2(#1)#3}
\Crefformat{equation}{#2(#1)#3}
\crefname{enumi}{}{}
\crefname{construction}{Construction}{Constructions}
\crefname{theoremA}{Theorem}{Theorems}

\usepackage{aliascnt}
\let\oldtheorem\newtheorem
\RenewDocumentCommand{\newtheorem}{s m o m O{}}{%
\IfBooleanTF{#1}%
{\oldtheorem{#2}{#4}}%
{\IfNoValueTF{#3}{\oldtheorem{#2}{#4}[#5]}%
{\newaliascnt{#2}{#3}%
\oldtheorem{#2}[#2]{#4}%
\aliascntresetthe{#2}}}}

\theoremstyle{plain}
\newtheorem{theorem}{Theorem}[section]
\newtheorem{lemma}[theorem]{Lemma}
\newtheorem{proposition}[theorem]{Proposition}
\newtheorem{corollary}[theorem]{Corollary}

\newtheorem*{theorem*}{Theorem}
\newtheorem{theoremA}{Theorem}

\theoremstyle{definition}
\newtheorem{definition}[theorem]{Definition}
\newtheorem{remark}[theorem]{Remark}
\newtheorem{example}[theorem]{Example}
\newtheorem{notation}[theorem]{Notation}
\newtheorem*{notation*}{Notation}
\newtheorem{construction}[theorem]{Construction}

\makeatletter
\usepackage[normalem]{ulem}
\usepackage{contour}

\contourlength{0.5pt}

\makeatletter
\DeclareRobustCommand{\myuline}[2][0pt]{%
  \ifmmode
    \uline{\hphantom{#2}\kern-#1}%
    \kern#1%
    \mathllap{\mathpalette\my@cont@{#2}}%
  \else
    \uline{\phantom{#2}\kern-#1}%
    \kern#1%
    \llap{\contour{background}{#2}}%
  \fi
}
\newcommand{\my@cont@}[2]{\contour{background}{\mbox{$\m@th#1#2$}}}
\makeatother

\DeclareMathOperator{\End}{End}

\newcommand{\id}{\mathrm{id}}
\DeclareMathOperator*{\colim}{colim}

\DeclareMathOperator{\ev}{ev}

\DeclareMathOperator{\LMod}{LMod}

\DeclareMathOperator{\Mod}{Mod}
\DeclareMathOperator{\Alg}{Alg}
\DeclareMathOperator{\coAlg}{coAlg}

\DeclareMathOperator{\Map}{Map}
\newcommand{\Sym}{M}

\DeclareMathOperator{\Ar}{Ar}

\DeclareMathOperator{\CMon}{CMon}
\newcommand{\Monad}{\mathrm{Monad}}
\newcommand{\whisk}{\cdot}

\newcommand{\hide}[1]{}

\newlist{numberenum}{enumerate}{1}
\setlist[numberenum]{\upshape(\arabic*)}

\newcommand{\Fun}{\mathrm{Fun}}
\newcommand{\twomap}{\myuline[1pt]{\mathrm{Map}}}

\newcommand{\activ}{\mathrm{act}}

\newcommand{\Fin}{\mathcal{F}\mathrm{in}}

\newcommand{\presl}{\mathcal{P}\mathrm{r}^{\mathrm{L}}}
\newcommand{\preslst}{\mathcal{P}\mathrm{r}^{\mathrm{L}}_{\mathrm{St}}}

\newcommand{\Spc}{\mathcal{S}}

\newcommand{\Cat}{\mathcal{C}\mathrm{at}}

\newcommand{\Catco}[1]{\Cat_{(\infty,1)}^{#1\mhyphen\mathrm{cocomp}}}

\newcommand{\LTop}{\mathcal{L}\mathrm{Top}}

\newcommand{\Cocart}{\mathcal{C}\mathrm{ocart}}

\DeclareMathOperator{\Env}{Env}

\newcommand{\op}{\mathrm{op}}
\newcommand{\coop}{\mathrm{coop}}

\newcommand{\Ind}{\mathrm{Ind}}

\newcommand{\Op}{\mathcal{O}\mathrm{p}}

\mathchardef\mhyphen="2D

\newcommand{\cocomp}{\mathrm{cocomp}}

\newcommand{\lex}{\mathrm{lex}}

\newcommand{\lax}{\mathrm{lax}}

\newcommand{\unit}{\mathds{1}}

\newcommand{\cA}{\mathcal{A}}
\newcommand{\cB}{\mathcal{B}}
\newcommand{\cC}{\mathcal{C}}
\newcommand{\cD}{\mathcal{D}}
\newcommand{\cE}{\mathcal{E}}

\newcommand{\cP}{\mathcal{P}}
\newcommand{\cQ}{\mathcal{Q}}

\newcommand{\cU}{\mathcal{U}}
\newcommand{\cV}{\mathcal{V}}

\newcommand{\cX}{\mathcal{X}}
\newcommand{\cY}{\mathcal{Y}}

\newcommand{\sX}{\mathcal{X}}
\newcommand{\sY}{\mathcal{Y}}
\newcommand{\sZ}{\mathcal{Z}}

\makeatletter
\newcommand{\oset}[3][0ex]{%
	\mathrel{\mathop{#3}\limits^{
			\vbox to#1{\kern-2\ex@
				\hbox{$\scriptstyle#2$}\vss}}}}
\makeatother
\newcommand{\eqarrow}{\oset[-.16ex]{\sim}{\longrightarrow}}
\newcommand{\eqarrowback}{\oset[-.16ex]{\sim}{\longleftarrow}}

\tikzset{
	rot90/.style={anchor=south, rotate=90, inner sep=.5mm}
}

\renewcommand{\subset}{\subseteq}

\renewcommand{\epsilon}{\varepsilon}

\def\lra{\longrightarrow}
\def\lla{\longleftarrow}

\def\llra{\def\arraystretch{.1}\begin{array}{c} \lra \\ \lla \end{array}}

\newcommand{\FatDelta}{\mathord{\text{\tikz[baseline=0pt,x=0.1em,y=0.1em]{%
  \node[inner sep=0,anchor=base west] at (0,0) {$\Delta$};
  \draw[line width=0.035em,line cap=butt] (2.67,0.5)--(5.215,5.6);}}}}

\usepackage{silence}
\DeclareRobustCommand{\SkipTocEntry}[5]{}

\begin{document}
\colorlet{background}{white}
\title{Lax-idempotent monads in homotopy theory}

\author{Fernando Abellán}
\address{Universität Regensburg,
Fakultät für Mathematik,
Universitätsstraße 31
93053 Regensburg, Germany}
\email{fernando.abellan@ur.de}

\author{Thomas Blom}
\address{Karlsruhe Institute of Technology, Englerstraße 2, 76131 Karlsruhe, Germany}
\email{thomas.blom@kit.edu}

\date{\today}

\begin{abstract}
	We extend the theory of lax-idempotent monads to the setting of $(\infty,2)$-categories.
	In particular, we prove an analogue of Kock and Zöberlein's characterization of the category of algebras over a lax-idempotent monad.
	We then apply this theory to many examples that arise naturally in homotopy theory.
	For example, we show that $(\infty,1)$-toposes are monadic over presentable $(\infty,1)$-categories, with left adjoint given by the \emph{symmetric $(\infty,1)$-topos}.
	We also show that $\infty$-operads are comonadic over symmetric monoidal $(\infty,1)$-categories through the monoidal envelope construction, obtaining along the way a new proof of Barkan--Steinebrunner's result that the envelope functor is a faithful inclusion.
	We furthermore use this theory to simplify Ramzi's proof that dualizable stable $(\infty,1)$-categories are comonadic over $\aleph_1$-compactly generated stable $(\infty,1)$-categories.
\end{abstract}

\maketitle
\thispagestyle{empty}

\vspace{-5pt}
\begin{figure}[h]
	\includegraphics[width=\textwidth]{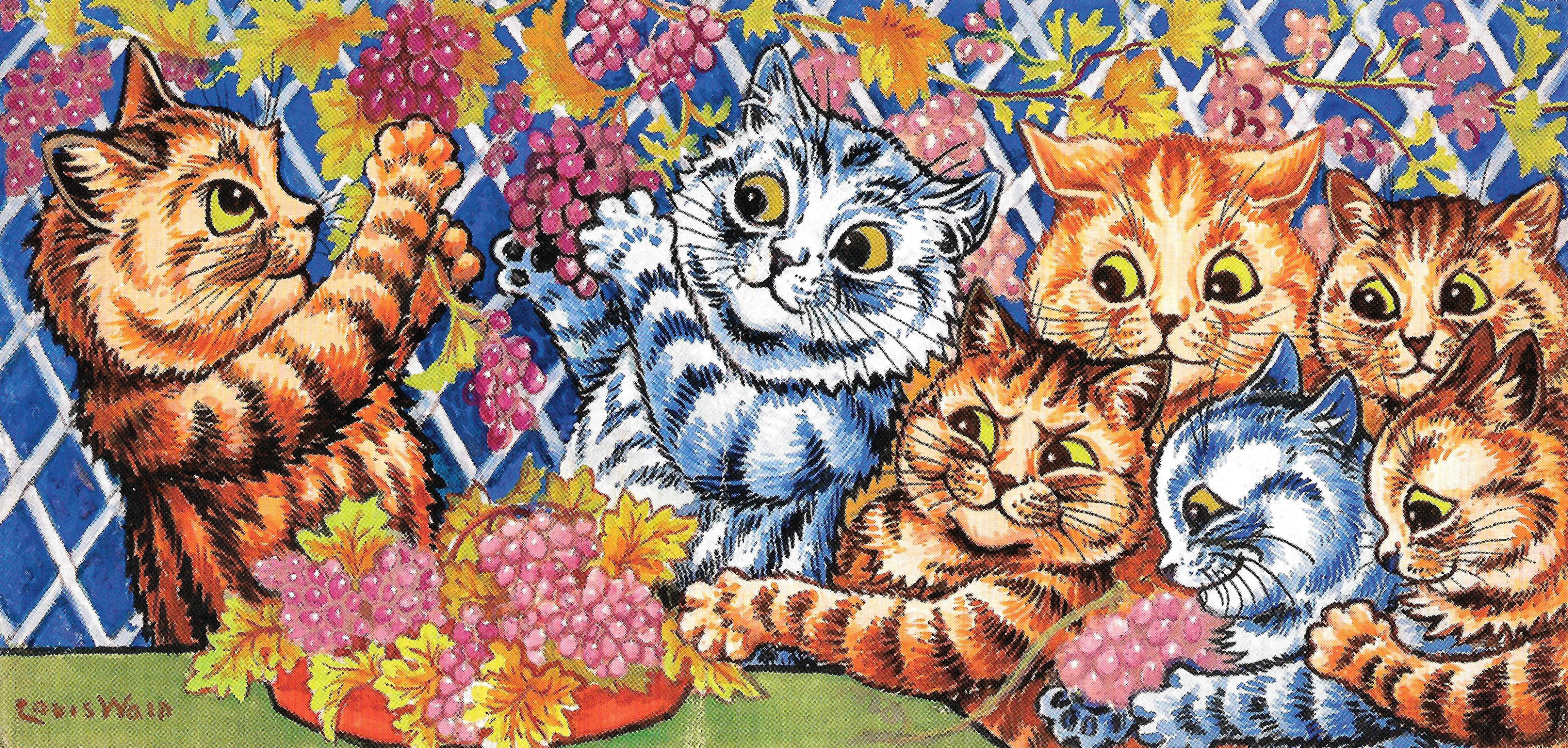}
	\caption[Cover Caption]{\textit{In the Vineyard}, Louis Wain\footnotemark}
\end{figure}
\footnotetext{\url{https://catland.distin.org/post/view/4181} (visited on 07/13/2026)}

\newpage
\tableofcontents

\addtocontents{toc}{\SkipTocEntry}
\section*{Introduction}

Among all monads, the \emph{idempotent} ones are particularly simple: their $(\infty,1)$-categories of algebras form a full subcategory of the original $(\infty,1)$-category, and the data of an idempotent monad on $\cC$ is essentially equivalent to that of an adjunction
\[\begin{tikzcd}[ampersand replacement=\&]
	{L : \cC} \&\& {\cD : R}
	\arrow[""{name=0, anchor=center, inner sep=0}, shift left=2, from=1-1, to=1-3]
	\arrow[""{name=1, anchor=center, inner sep=0}, shift left=2, hook', from=1-3, to=1-1]
	\arrow["\dashv"{anchor=center, rotate=-90}, draw=none, from=0, to=1]
\end{tikzcd}\]
whose right adjoint is fully faithful.

However, there are many examples of adjunctions in nature where the right adjoint is not quite fully faithful, but it is still an inclusion of a subcategory:\footnote{We say that a functor $F \colon \cC \to \cD$ is an \emph{inclusion} (of a subcategory) if it is faithful and the map on groupoid cores $\cC^\simeq \to \cD^\simeq$ is an inclusion of path components; this is sometimes called a \emph{replete subcategory}.}
\begin{enumerate}[(1)]
	\item\label{introex1} The inclusion $\Cat_{(\infty,1)}^\mathrm{rex} \hookrightarrow \Cat_{(\infty,1)}$ of the $(\infty,1)$-category of finitely cocomplete $(\infty,1)$-categories and right exact functors into the $(\infty,1)$-category of all $(\infty,1)$-categories.
	Its left adjoint sends an $(\infty,1)$-category $\cC$ to the full subcategory of $\cP(\cC)$ spanned by finite colimits of representables \cite[\S 5.3.6]{HTT}.
	\item\label{introex2} The inclusion $\CMon(\Cat_{(\infty,1)}) \hookrightarrow \Op_\infty$ of the $(\infty,1)$-category of symmetric monoidal $(\infty,1)$-categories (and strong monoidal functors between them) into the $(\infty,1)$-category of $\infty$-operads.
	Its left adjoint is the symmetric monoidal envelope construction \cite[\S 2.2.4]{HA}.
	\item\label{introex3} For a fixed $(\infty,1)$-category $\cB$, the inclusion $\Cocart(\cB) \hookrightarrow \Cat_{(\infty,1)/\cB}$ of cocartesian fibrations over $\cB$ into all $(\infty,1)$-categories over $\cB$. Its left adjoint is the free fibration \cite{GepnerHaugsengea2017LaxColimitsFree}.
	\item\label{introex4} The inclusion $\LTop \hookrightarrow \presl$ of the $(\infty,1)$-category of $(\infty,1)$-toposes and left exact left adjoints into the $(\infty,1)$-category of all presentable $(\infty,1)$-categories.
	Its left adjoint is the \emph{symmetric $(\infty,1)$-topos}, constructed in \cref{sec:symtop} below.
	\item\label{introex5} The inclusion $(\preslst)^\mathrm{dbl} \hookrightarrow \presl_{\mathrm{St},\aleph_1}$ of dualizable stable $(\infty,1)$-categories into $\aleph_1$-compactly generated stable $(\infty,1)$-categories is a left adjoint, whose right adjoint is given by $\Ind((-)^{\aleph_1})$ (see \cite[Proposition 1.89]{efimovKtheoryLocalizingInvariants2025}).
\end{enumerate}

A feature that all these examples have in common is that they are most naturally viewed as $(\infty,2)$-categories, and these inclusions are \emph{locally fully faithful}, meaning that they induce fully faithful functors on mapping $(\infty,1)$-categories.
A natural question is thus whether there exists, for monads $T$ on an $(\infty,2)$-category $\sX$, a weakening of the notion of an idempotent monad that ensures that the forgetful functor $\Alg_T(\sX) \to \sX$ is a locally fully faithful inclusion.

In classical category theory, this question was answered by Kock \cite{OG-Kock} and further studied by Zöberlein \cite{Zoberlein1976Doctrines2categories}:
they defined a (pseudo)monad $T$ on a $(2,2)$-category $\sX$ to be \emph{lax-idempotent}\footnote{In the classical category theory literature, the names ``KZ-doctrine'' or ``KZ-(pseudo)monad'' are also common for what we call a lax-idempotent monad.} if for any $x$ in $\sX$, the multiplication map $\mu_x \colon TTx \to Tx$ is left adjoint to the unit $\eta_{Tx} \colon Tx \to TTx$.
For such a monad $T$, they proved an explicit description of $\Alg_T(\sX)$ as a locally full sub-$(2,2)$-category of $\sX$.
Our main result is an extension of this description to the higher-categorical setting.

\begin{theoremA}\label{thmA}
	Let $T$ be a monad on an $(\infty,2)$-category $\sX$ and suppose that for any $x$ in $\sX$, the multiplication map $\mu_x \colon TTx \to Tx$ is left adjoint to the unit $\eta_{Tx} \colon Tx \to TTx$.
	Then $\Alg_T(\sX) \to \sX$ is a locally fully faithful inclusion of $(\infty,2)$-categories, and its image is described as follows:
	\begin{enumerate}
		\item Its objects are precisely those $x$ in $\sX$ for which the unit map $\eta_x \colon x \to Tx$ admits a left adjoint $\alpha_x$ such that $\alpha_x \eta_x \simeq \id_x$.
		\item Its 1-morphisms are precisely those $f \colon x \to y$ for which the Beck--Chevalley transformation
		\[\begin{tikzcd}
			Tx & Ty \\
			x & y
			\arrow["Tf", from=1-1, to=1-2]
			\arrow["{\alpha_x}"', from=1-1, to=2-1]
			\arrow[shorten <= 6, shorten >= 6, Rightarrow, from=1-2, to=2-1]
			\arrow["{\alpha_y}", from=1-2, to=2-2]
			\arrow["f"', from=2-1, to=2-2]
		\end{tikzcd}\]
		is invertible.
	\end{enumerate}
\end{theoremA}

In light of this result, we will call a monad $T$ on an $(\infty,2)$-category such that $\mu_x$ is left adjoint to $\eta_{Tx}$ \emph{lax-idempotent}.
We also establish a converse to \cref{thmA}: if
\[\begin{tikzcd}[ampersand replacement=\&]
	{F : \cX} \&\& {\cY : U}
	\arrow[""{name=0, anchor=center, inner sep=0}, shift left=2, from=1-1, to=1-3]
	\arrow[""{name=1, anchor=center, inner sep=0}, shift left=2, from=1-3, to=1-1]
	\arrow["\dashv"{anchor=center, rotate=-90}, draw=none, from=0, to=1]
\end{tikzcd}\]
is an adjunction between $(\infty,2)$-categories such that $\cY$ looks like the description of $\Alg_T(\sX)$ given in \cref{thmA} above, then $U$ is monadic and the monad $UF$ is lax-idempotent; see \cref{main-thm-inverse-2} for a precise statement.

Part of the value of \cref{thmA} lies in the fact that it is generally very useful to know that an $(\infty,1)$-category (or, in this case, an $(\infty,2)$-category) is monadic over another one. For example, it can be used to establish that an $(\infty,1)$-category is presentable, it allows one to build canonical free resolutions, it provides techniques for computing limits and colimits, etc.
\Cref{thmA} provides a very simple criterion to check whether a right adjoint is monadic, in situations where the Barr--Beck--Lurie theorem \cite[Theorem 4.7.3.5]{HA} may not be so easy to apply.
For instance, in the examples described above, geometric realizations are generally hard to compute, making it hard to verify condition (c.1) of \cite[Theorem 4.7.3.5]{HA}.
However, the conditions of \cref{thmA} are relatively straightforward to verify for the examples \crefrange{introex1}{introex3}, as we will show in \crefrange{sec:cocart}{sec:cocompletions}.

The inclusion $\LTop \hookrightarrow \presl$ from \cref{introex4} is of a slightly different nature, because it corresponds to an \emph{oplax-idempotent monad}: these are formally dual to lax-idempotent monads, in the sense that one requires $\mu_x$ to be \emph{right} adjoint to $\eta_{Tx}$.
We will show the following higher-categorical analogue of a result of Bunge--Funk \cite[\S 4.2]{2006SingularCoveringsToposes} in \cref{sec:symtop}.

\begin{theoremA}\label{thmB}
	The inclusion $\LTop \hookrightarrow \presl$ admits a left adjoint $\Sym$, called the \emph{symmetric $(\infty,1)$-topos}.
	Moreover, this inclusion is monadic, and the corresponding monad $\Sym$ on $\presl$ is oplax-idempotent.
\end{theoremA}

The left adjoint $\Sym \colon \presl \to \LTop$ was also recently constructed by Lehner \cite[Theorem 18.33]{Lehner2026ShapeTheory}.

If one replaces the monad $T$ with a comonad, one arrives at the definition of an \emph{(op)lax-idempotent comonad}.
\Cref{thmA} has obvious analogues for (op)lax-idempotent (co)monads, which we spell out explicitly in \cref{main-thm-2}.
We also discuss two important examples of lax-idempotent comonads in this paper.
The first of these concerns the inclusion $\CMon(\Cat_{(\infty,1)}) \hookrightarrow \Op_\infty$ mentioned in \cref{introex2}: it turns out that not only is this inclusion monadic, but also that its left adjoint, the envelope construction, is comonadic.

\begin{theoremA}\label{thmC}
	The envelope construction $\Env \colon \Op_\infty \to \CMon(\Cat_{(\infty,1)})$ is comonadic, and the corresponding comonad on $\CMon(\Cat_{(\infty,1)})$ is lax-idempotent.
	In particular, $\Env \colon \Op_\infty \to \CMon(\Cat_{(\infty,1)})$ is a locally fully faithful inclusion of $(\infty,2)$-categories.
\end{theoremA}

The faithfulness of $\Env$ was also proved by Barkan--Steinebrunner \cite[Theorem 3.2.13]{BarkanSteinebrunner2026EquifiberedApproachProperads} through different means.

The second example concerns the inclusion $(\preslst)^\mathrm{dbl} \hookrightarrow \presl_{\mathrm{St},\aleph_1}$ described in \cref{introex5}: we show that this inclusion is comonadic, recovering a result of Ramzi, and that the corresponding comonad $\Ind((-)^{\aleph_1})$ is lax-idempotent.
We also describe how this example is related to the general notion of a \emph{continuous algebra} over a lax-idempotent monad.
Our comonadicity proof should not be seen as independent of Ramzi's proof, though; we use the fact that any dualizable stable $(\infty,1)$-category is $\aleph_1$-compactly generated as input.

\subsection*{Outlook}
The notion of a lax-idempotent monad can easily be generalized: given a monoidal $(\infty,2)$-category $\cV$, a \emph{lax-idempotent algebra} in $\cV$ is an associative algebra $A \in \Alg(\cV)$ with the property that the multiplication map $A \otimes A \to A$ is left adjoint to the unit $\eta \otimes \id_A \colon A \to A \otimes A$.
Lax-idempotent monads are a special case of these: they are simply lax-idempotent algebras in $\Fun(\sX,\sX)$.
We therefore state some of our results at the level of lax-idempotent algebras (see \cref{sec:laxi-algebras}) and later specialize to the case of lax-idempotent monads.

In the companion paper \cite{companion}, we will further develop the theory of lax-idempotent algebras and show that they have many properties resembling those of idempotent algebras.
For example, we show that the full sub-$(\infty,2)$-category $\Alg^\mathrm{li}(\cV)$ of $\Alg(\cV)$ spanned by the lax-idempotent algebras is a poset, and that any lax-idempotent algebra $A$ is fully determined by its unit map $\unit \to A$, in the sense that the forgetful functor $\Alg^\mathrm{li}(\cV) \to \cV_{\unit/}$ is a faithful inclusion.
In particular, all the coherences of a lax-idempotent algebra come ``for free''.

We will also prove an $(\infty,2)$-categorical version of a classical observation of Street \cite[(2.27)]{Street1980FibrationsBicategories}.
Recall that monads on $\sX$ may be identified with strong monoidal functors $\Delta_+ \to \Fun(\sX,\sX)$.
Here $\Delta_+$ denotes the category of finite linearly ordered sets, with its monoidal structure given by concatenation.
The category $\Delta_+$ naturally refines to a monoidal $(2,2)$-category $\FatDelta_+$ (or, more precisely, a poset-enriched category) by declaring that there exists a 2-morphism $f \Rightarrow g$ precisely if $f \leq g$ in the pointwise ordering.
We will show that a monad $T$ on an $(\infty,2)$-category $\sX$ is lax-idempotent precisely if the corresponding strong monoidal functor $\Delta_+ \to \Fun(\sX,\sX)$ can be extended to a strong monoidal functor $\FatDelta_+ \to \Fun(\sX,\sX)$.
More generally, we will show that for any monoidal $(\infty,2)$-category $\cV$, the full sub-$(\infty,2)$-category inclusion $\Alg^\mathrm{li}(\cV) \hookrightarrow \Alg(\cV)$ agrees with the functor
\[\Fun^\otimes(\FatDelta_+,\cV) \to \Fun^\otimes(\Delta_+,\cV)\]
defined by restriction along $\Delta_+ \to \FatDelta_+$. In other words, we will identify the delooping of $\FatDelta_{+}$ with the walking lax-idempotent monad. Moreover, we will give a description of the walking lax-idempotent adjunction in \cite{companion}, a notion which we define in \cref{def:laxi-adjunctions} below.

\subsection*{Overview of the paper}

This paper is split into two parts: first we develop the basic theory of lax-idempotent monads on $(\infty,2)$-categories, and in the second part we discuss all the examples mentioned above.
We discuss some preliminaries on adjunctions in $(\infty,2)$-categories in \cref{sec:prelims-adj}; here we also collect some useful facts that we found hard to locate in the literature.
We then introduce lax-idempotent algebras in monoidal $(\infty,2)$-categories in \cref{sec:laxi-algebras} and prove their basic properties.
After developing the basic theory of monads on $(\infty,2)$-categories in \cref{sec:heine}, we introduce lax-idempotent monads in \cref{sec:limonad} as lax-idempotent algebras in $\Fun(\sX,\sX)$.
\Cref{sec:talg} is the technical heart of the paper: here we prove our main result \cref{thmA} and its converse, \cref{main-thm-inverse-2}.
We conclude the first part by proving that lax-idempotent monads on $(\infty,2)$-categories give rise to adjoint functor theorems, following \cite{ArkorLibertiea2024AdjointFunctorTheorems}.

In the second part, we discuss the following examples in detail:
\begin{itemize}
	\item Cocartesian fibrations of $(\infty,1)$-categories.
	\item $\infty$-operads and symmetric monoidal $(\infty,1)$-categories. Here we also prove \cref{thmC}.
	\item (Co)completions of $(\infty,1)$-categories for a given class of colimit shapes.
	\item Compactly assembled and dualizable $(\infty,1)$-categories.
	This also includes a general discussion of \emph{continuous algebras} over a lax-idempotent monad.
	\item Presentable $(\infty,1)$-categories and $(\infty,1)$-toposes. Here we prove \cref{thmB}.
\end{itemize}

\subsection*{Use of large language models (LLMs)} LLMs were used for proofreading the paper and generating pre-submission ``referee reports''. These reports are available upon request. No LLMs were used for generating mathematical ideas or content or for any of the writing.

\subsection*{Acknowledgements}

We would like to thank Adam Dauser, Georg Lehner, Tobias Lenz, Sil Linskens, Louis Martini, Lyne Moser, Phil Pützstück, Maxime Ramzi and Jan Steinebrunner for conversations related to this paper and their interest in the project.
We also thank Phil Pützstück for the observation that any lax-idempotent algebra in a braided monoidal $(\infty,2)$-category is idempotent (\cref{prop:Phil}).

We are grateful to the Max Planck Institute for Mathematics in Bonn for its hospitality and support during the writing of this paper.
The second author is also grateful to the Copenhagen Centre for Geometry and Topology (DNRF151), where this project was started. He was partially supported by the European Union through an ERC grant (MaFC, 101221003).

\subsection*{Notation and terminology}
\begin{itemize}
	\item Given an $(\infty,2)$-category $\sX$, we write $\twomap_\sX(x,y)$ or $\twomap(x,y)$ for the $(\infty,1)$-category of maps from $x$ to $y$.
	\item Given $(\infty,2)$-categories $\sX$ and $\sY$, $\Fun(\sX,\sY)$ will denote the $(\infty,2)$-category of functors from $\sX$ to $\sY$, i.e.\ the internal hom of $\Cat_{(\infty,2)}$.
	\item We use $- \whisk -$ to denote the whiskering of 1-morphisms with 2-morphisms.
	\item Given an $(\infty,2)$-category $\sX$, we write $\sX^\op$ and $\sX^\mathrm{co}$ for the $(\infty,2)$-categories obtained by reversing the 1-morphisms and 2-morphisms of $\sX$, respectively.
\end{itemize}

\part{The theory of lax-idempotent monads}

\section{Some generalities on adjunctions in \texorpdfstring{$(\infty,2)$}{(∞,2)}-categories}\label{sec:prelims-adj}

We will assume the reader is familiar with the theory of $(\infty,2)$-categories and use its language freely in a model-independent way.
For introductions to the theory of $(\infty,2)$-categories, see for example \cite[\S 2]{Haugseng2021LaxTransformationsAdjunctions} and \cite[Appendix A]{BlansBlom2024ChainRuleGoodwillie} and the references cited there.
The purpose of this section is to collect a few facts about adjunctions in $(\infty,2)$-categories that we found hard to locate in the literature.
Let us first define \emph{whiskerings}:

\begin{definition}[Whiskerings]
	Let $\sX$ be an $(\infty,2)$-category and suppose we are given a diagram
	\[\begin{tikzcd}
		x & y & z & w
		\arrow["f"', from=1-1, to=1-2]
		\arrow[""{name=0, anchor=center, inner sep=0}, "g", curve={height=-12pt}, from=1-2, to=1-3]
		\arrow[""{name=1, anchor=center, inner sep=0}, "{g'}"', curve={height=12pt}, from=1-2, to=1-3]
		\arrow["h"', from=1-3, to=1-4]
		\arrow["\alpha", between={0.2}{0.8}, Rightarrow, from=0, to=1]
	\end{tikzcd}\]
	in $\sX$.
	Then one can compose at either $y$ or $z$ to form the \emph{whiskerings} $\alpha \whisk f \colon gf \Rightarrow g'f$ and $h \whisk \alpha \colon hg \Rightarrow hg'$, respectively.
	Concretely, the 2-morphisms $\alpha \whisk f$ and $h \whisk \alpha$ are defined as the images of $(\id_f, \alpha)$ and $(\alpha, \id_h)$ under the composition functors
	\[\twomap(x,y) \times \twomap(y,z) \to \twomap(x,z) \quad \text{and} \quad \twomap(y,z) \times \twomap(z,w) \to \twomap(y,w)\]
	respectively.
	We will sometimes omit the $\whisk$ and simply write $\alpha f$ and $h \alpha$.
\end{definition}

\begin{definition}[Adjunctions in an $(\infty,2)$-category]\label{def:adj}
	Given 1-morphisms $f \colon x \to y$ and $g \colon y \to x$ in an $(\infty,2)$-category $\sX$, we say that $f$ is left adjoint to $g$, denoted $f \dashv g$, if there exist 2-morphisms $\eta \colon \id_x \Rightarrow gf$ and $\varepsilon \colon fg \Rightarrow \id_y$ such that there exist homotopies making the triangles
	\[\begin{tikzcd}[ampersand replacement=\&]
		f \& fgf \\
		\& f
		\arrow["{f \whisk \eta}", Rightarrow, from=1-1, to=1-2]
		\arrow["{\id_f}"', Rightarrow, from=1-1, to=2-2]
		\arrow["{\varepsilon \whisk f}", Rightarrow, from=1-2, to=2-2]
	\end{tikzcd}
	\quad \text{and} \quad
	\begin{tikzcd}[ampersand replacement=\&]
		g \& gfg \\
		\& g
		\arrow["{\eta \whisk g}", Rightarrow, from=1-1, to=1-2]
		\arrow["{\id_g}"', Rightarrow, from=1-1, to=2-2]
		\arrow["{g \whisk \varepsilon}", Rightarrow, from=1-2, to=2-2]
	\end{tikzcd}\]
	commute.
	Observe that these \emph{triangle identities} can be checked in the homotopy $(2,2)$-category of $\sX$.
\end{definition}

We will now prove several facts about \emph{reflective} adjunctions that are used throughout this paper.

\begin{definition}[Reflective localizations]
	Let $f : y \rightleftarrows x : g$ be an adjunction in an $(\infty,2)$-category $\sX$.
	We will call this adjunction \emph{reflective}, and call $f$ a \emph{reflective localization} or \emph{reflective left adjoint}, if the counit $\varepsilon \colon fg \Rightarrow \id_x$ is an equivalence.
	If instead the unit $\eta \colon \id_y \Rightarrow gf$ is an equivalence, then we call this adjunction \emph{coreflective}, and we call $g$ a \emph{coreflective localization} or \emph{coreflective right adjoint}.
\end{definition}

They can be recognized as follows.

\begin{lemma}\label{easyadjun}
	Let $\sX$ be an $(\infty,2)$-category and $f \colon x \to y$ and $g \colon y \to x$ be 1-morphisms in $\sX$ together with an equivalence $fg \simeq \id_y$.
	Then this equivalence is the counit of a (necessarily reflective) adjunction $f \dashv g$ if and only if there exists a 2-morphism $\nu \colon \id_x \Rightarrow gf$ such that the 2-morphisms $f \whisk \nu$ and $\nu \whisk g$ are invertible.
\end{lemma}

\begin{proof}
	This follows by the same argument as \cite[Lemma 2.1.11]{RiehlVerity2022ElementsCategoryTheory}.
	Note that in the proof, Riehl--Verity modify the counit of the adjunction while leaving the unit intact, but by a dual argument one can leave the counit $fg \simeq \id_y$ intact while modifying $\nu$ to become a unit.
\end{proof}

We now give several equivalent characterizations of reflective adjunctions.
Recall that a 1-morphism $f \colon x \to y$ in an $(\infty,2)$-category $\sX$ is called \emph{fully faithful} if for any $z$ in $\sX$, the map
\[f_* \colon \twomap_{\sX}(z,x) \to \twomap_{\sX}(z,y)\]
is fully faithful.

\begin{lemma}\label{lem:lali}
	Let $f : y \rightleftarrows x : g$ be an adjunction in an $(\infty,2)$-category $\sX$.
	Then the following statements are equivalent:
	\begin{enumerate}[(a)]
		\item\label{lali1} The adjunction is reflective (i.e. the counit $\varepsilon \colon fg \Rightarrow \id_x$ is an equivalence).
		\item\label{lali2} The composite $fg \colon x \to x$ is an equivalence.
		\item\label{lali3} The right adjoint $g$ is fully faithful in $\sX$.
	\end{enumerate}
\end{lemma}

\begin{proof}
	The equivalence of \cref{lali1} and \cref{lali2} follows as in \cite[Lemma 3.3.1]{CarmeliSchlankea2021AmbidexterityHeight}, noting that their proof goes through for adjunctions in an arbitrary $(\infty,2)$-category.

	The equivalence of \cref{lali1} and \cref{lali3} is well known when $\sX = \Cat_{(\infty,1)}$.
	For the general case, observe that since $\twomap_{\sX}(z,-)$ is a 2-functor, it preserves adjunctions whose counit is an equivalence. This shows that \cref{lali1} implies \cref{lali3}.
	For the converse implication, note that the counit component of $f_* : \twomap_{\sX}(x,y) \rightleftarrows \twomap_{\sX}(x,x) : g_*$ at $\id_x$ is given by $\varepsilon \colon fg = f_*g_*(\id_x) \Rightarrow \id_x$, which is invertible since $g_*$ is fully faithful.
\end{proof}

Let us also record the following standard consequences.

\begin{lemma}\label{lem:lali-dk-loc}
	Let $f : \cC \rightleftarrows \cD : g$ be an adjunction between $(\infty,1)$-categories.
	Then $f$ is a reflective localization if and only if it is a Dwyer--Kan localization.
\end{lemma}

\begin{proof}
	One direction is given by \cite[Proposition 5.2.7.12]{HTT}, the other follows from \cref{lem:lali} since any Dwyer--Kan localization is fully faithful in $\Cat_{(\infty,1)}^\op$.
\end{proof}

\begin{lemma}\label{lem:extra-adj-ff}
	Let $f : y \rightleftarrows x : g$ be an adjunction in an $(\infty,2)$-category $\sX$ and suppose that $f$ admits a further left adjoint $q$.
	Then $g$ is fully faithful if and only if $q$ is fully faithful.
	Equivalently, the adjunction $f \dashv g$ is reflective if and only if the adjunction $q \dashv f$ is coreflective.
\end{lemma}

\begin{proof}
	By applying \cref{lem:lali} to $\sX^\op$ and $\sX^\coop$, it follows that both conditions are equivalent to $f$ being fully faithful in $\sX^\op$.
\end{proof}

\begin{lemma}\label{BCeasy}
	Let $\sX$ be an $(\infty,2)$-category and
	\[\begin{tikzcd}
		x & {x'} \\
		y & {y'}
		\arrow["f",from=1-1, to=1-2]
		\arrow["i"', from=1-1, to=2-1]
		\arrow["{j}", from=1-2, to=2-2]
		\arrow["g"',from=2-1, to=2-2]
	\end{tikzcd}\]
	a commutative square in $\sX$.
	Suppose $i$ and $j$ admit reflective left adjoints $k$ and $l$.
	Then the Beck--Chevalley transformation
	\[\begin{tikzcd}
		x & {x'} \\
		y & {y'}
		\arrow["f", from=1-1, to=1-2]
		\arrow[shorten <= 6, shorten >= 6, Rightarrow, from=2-2, to=1-1]
		\arrow["k", from=2-1, to=1-1]
		\arrow["g"', from=2-1, to=2-2]
		\arrow["l"', from=2-2, to=1-2]
	\end{tikzcd}\]
	is invertible if and only if there exists some homotopy $lg \simeq fk$ making this square commute.
\end{lemma}

\begin{proof}
	The Beck--Chevalley transformation is the composite
	\[lg \xRightarrow{\eta} lgik \simeq ljfk \xRightarrow{\epsilon} fk.\]
	Since $\varepsilon$ is invertible by assumption, it suffices to show that $lg \whisk \eta \colon lg \Rightarrow lgik$ is invertible.
	Because of the homotopy $lg \simeq fk$, this holds if $fk \whisk\eta \colon fk \Rightarrow fkik$ is invertible.
	Since the 2-morphism $k \whisk \eta \colon k \Rightarrow kik$ is invertible by the triangle identities, the result follows.
\end{proof}

\begin{proposition}\label{lem:lift-along-lali}
	Let
	\[\begin{tikzcd}[ampersand replacement=\&]
		x \& y \& z
		\arrow["g"', shift right=1, hook, from=1-1, to=1-2]
		\arrow["f"', shift right=1, from=1-2, to=1-1]
		\arrow["h", from=1-2, to=1-3]
	\end{tikzcd}\]
	be 1-morphisms in an $(\infty,2)$-category $\sX$ such that $f$ is a reflective left adjoint to $g$.
	Suppose furthermore $h \whisk \eta \colon h \Rightarrow hgf$ is invertible, where $\eta$ is the unit of $f \dashv g$.
	Then $h$ admits a right adjoint if and only if $hg$ admits a right adjoint.
\end{proposition}

\begin{proof}
	For the ``if'' direction, suppose that $k$ is right adjoint to $hg$.
	Then $gk$ is right adjoint to $hgf \simeq h$.
	Since the triangle identities can be checked in the homotopy $(2,2)$-category of $\sX$, the converse follows from \cite[Lemma 2.7]{ArkorLibertiea2024AdjointFunctorTheorems}.
\end{proof}

\section{Lax-idempotent algebras in a monoidal \texorpdfstring{$(\infty,2)$}{(∞,2)}-category}\label{sec:laxi-algebras}

We will now define lax-idempotent algebras in the $(\infty,2)$-categorical setting.
Lax-idempotent monads on an $(\infty,2)$-category $\sX$ will be defined as lax-idempotent algebras in $(\Fun(\sX,\sX),\circ)$ in \cref{sec:limonad}.

The following two definitions are essentially the same as for $(\infty,1)$-categories, but we include them for completeness' sake.

\begin{definition}[Monoidal $(\infty,2)$-categories]\label{def:monoidal-2cat}
	A \emph{monoidal $(\infty,2)$-category} is a functor $\sX_\bullet \colon \Delta^\op \to \Cat_{(\infty,2)}$ such that for every $n \geq 0$, the Segal map
	\[\sX_n \to \sX_1 \times \cdots \times \sX_1\]
	is an equivalence of $(\infty,2)$-categories.
\end{definition}

We will generally write $\sX$ for $\sX_1$.
The composite $\sX \times \sX \eqarrowback \sX_2 \to \sX_1$ will be denoted $- \otimes -$, and we will write $\unit$ for the object $* \eqarrowback \sX_0 \to \sX_1$.
Given a monoidal $(\infty,2)$-category $\sX$, we will write $\sX^\otimes \to \Delta^\op$ for the corresponding cocartesian fibration under the 2-categorical straightening equivalence \cite{Nuiten2024SuperStraight} (see also \cite[Proposition A.2.9]{BlansBlom2024ChainRuleGoodwillie} for a simplified straightening equivalence if the base is an $(\infty,1)$-category).

\begin{definition}[Algebras in a monoidal $(\infty,2)$-category]\label{def:2alg}
	Let $\sX$ be a monoidal $(\infty,2)$-category.
	Then an \emph{algebra} in $\sX$ is defined as a section of $\sX^\otimes \to \Delta^\op$ that takes inert morphisms to cocartesian morphisms.
	We write $\Alg(\sX)$ for the full sub-$(\infty,2)$-category of
	\[\Fun_{/\Delta^\op}(\Delta^\op,\sX^\otimes) = \Fun(\Delta^\op,\sX^\otimes) \times_{\Fun(\Delta^\op,\Delta^\op)} \{\mathrm{id}\}\]
	spanned by the algebras.
\end{definition}

Given an algebra $A \colon \Delta^\op \to \sX^\otimes$, we abusively write $A$ for its value on $[1]$.
Since it preserves cocartesian lifts of inerts, it takes $[n]$ to $(A,\ldots,A) \in \sX^{\times n} \simeq \sX^\otimes_n$.
In particular, the map $d_1 \colon [1] \to [2]$ in $\Delta$ induces the multiplication map $\mu \colon A \otimes A \to A$, while $s_0 \colon [1] \to [0]$ induces the unit map $u \colon \unit \to A$.

\begin{definition}[Lax-idempotent algebras]\label{Laxialg}
	Let $A$ be an algebra in a monoidal $(\infty,2)$-category $\sX$.
	Then $A$ is called \emph{lax-idempotent} if the multiplication map $\mu \colon A \otimes A \to A$ is left adjoint to $u \otimes \id_A \colon A \to A \otimes A$.
\end{definition}

\begin{lemma}
	Let $A$ be a lax-idempotent algebra. Then $\mu \colon A \otimes A \to A$ is a reflective localization.
\end{lemma}

\begin{proof}
	The definition of an algebra yields an equivalence $\mu \circ (u \otimes \id_A) \simeq \id_A$, hence this follows from \cref{lem:lali}.
\end{proof}

\begin{remark}
	By \cref{easyadjun}, one may assume without loss of generality that the equivalence $\mu \circ (u \otimes \id_A) \simeq \id_A$ coming from the algebra structure on $A$ is the counit of the adjunction $\mu \dashv (u \otimes \id_A)$.
\end{remark}

By reversing the directions of the 1- or 2-morphisms, we obtain three more variants of \cref{Laxialg}.
We define a \emph{coalgebra} in $\sX$ to be an algebra in $\sX^\op$.

\begin{definition}[(Op)lax-idempotent (co)algebras]\label{LaxialgVar}
Let $A$ be an algebra in a monoidal $(\infty,2)$-category $\sX$.
Then $A$ is called \emph{oplax-idempotent} if the multiplication $\mu$ is right adjoint to the unit $u \otimes \id_A$.
A coalgebra $Q$ in $\sX$ is called \emph{lax-idempotent} (resp.\ \emph{oplax-idempotent}) if the comultiplication $\delta \colon Q \to Q \otimes Q$ is left (resp.\ right) adjoint to $c \otimes \id_Q \colon Q \otimes Q \to Q$, where $c$ denotes the counit of the coalgebra.
\end{definition}

For brevity, in the rest of the paper we will mainly state our results for lax-idempotent algebras.
By reversing the 1- or 2-morphisms, one obtains analogous results for (op)lax-idempotent (co)algebras.

The following alternative characterizations of lax-idempotent algebras will be used throughout the paper.

\begin{proposition}\label{altAlg}
	Let $A$ be an algebra in a monoidal $(\infty,2)$-category $\sX$. Then the following are equivalent:
	\begin{enumerate}[(1)]
		\item\label{altAlg1} $A$ is lax-idempotent.
		\item\label{altAlg2} The multiplication $\mu \colon A \otimes A \to A$ is right adjoint to $\id_A \otimes u$.
		\item\label{altAlg3} There exists a 2-morphism $d \colon \id_A \otimes u \Rightarrow u\otimes \id_A$ such that $d \whisk u$ and $\mu \whisk d$ are invertible.
	\end{enumerate}
\end{proposition}

\begin{proof}
	To see that \cref{altAlg1} implies \cref{altAlg3}, let $d$ be the composite
	\begin{equation}\label{eq:modification-d}
	\id_A \otimes u \xRightarrow{\eta \whisk (\id_A \otimes u)} (u \otimes \id_A) \circ \mu \circ (\id_A \otimes u) \xRightarrow{(u \otimes \id_A) \whisk \sigma} u \otimes \id_A,
	\end{equation}
	where $\sigma$ is the equivalence $\mu \circ (\id_A \otimes u) \simeq \id_A$ coming from the algebra structure of $A$.
	To see that $d \whisk u$ is invertible, it suffices to show that the left 2-morphism in \cref{eq:modification-d} becomes invertible when whiskering with $u$.
	This holds since
	\[\eta \whisk ((\id_A \otimes u) \circ u) \simeq (\eta \whisk (u \otimes \id_A)) \whisk u\]
	and $\eta \whisk (u \otimes \id_A)$ is invertible by the triangle identities.
	It also follows from the triangle identities that $\mu \whisk \eta$, and hence $\mu \whisk d$, are invertible.

	Conversely, if \cref{altAlg3} holds then we may form the 2-morphism $\eta \colon \id_{A \otimes A} \Rightarrow (u \otimes \id_A) \circ \mu$ as the composite
	\[\id_{A \otimes A} \simeq (\id_A \otimes \mu) \circ (\id_A \otimes u \otimes \id_A) \xRightarrow{d} (\id_A \otimes \mu) \circ (u \otimes \id_A \otimes \id_A) \simeq  (u \otimes \id_A) \circ \mu.\]
	It follows similarly as above that $\eta \whisk (u \otimes \id_A)$ and $\mu \whisk \eta$ are invertible, hence $A$ is lax-idempotent by \cref{easyadjun}.

	The equivalence between \cref{altAlg2} and \cref{altAlg3} can be proved by an analogous argument, or by applying the equivalence \cref{altAlg1} $\Leftrightarrow$ \cref{altAlg3} to $\cX^\mathrm{co}$ with the reversed monoidal structure.
\end{proof}

\begin{remark}
	In \cite{companion}, we will study lax-idempotent algebras in more detail.
	In particular, we will show analogues of two classical facts about idempotent algebras:
	\begin{enumerate}
		\item The full sub-$(\infty,2)$-category $\Alg^{\mathrm{li}}(\sX) \subset \Alg(\sX)$ spanned by lax-idempotent algebras is a poset.
		More generally, if $A$ and $B$ are algebras in $\sX$ such that $A$ is lax-idempotent, then $\twomap_{\Alg(\sX)}(A,B)$ is either contractible or empty.
		\item A lax-idempotent algebra is fully determined by its unit map $\unit \to A$.
		More precisely, the forgetful functor $\Alg^\mathrm{li}(\sX) \to \sX_{\unit/}$ is a locally fully faithful inclusion.
	\end{enumerate}
\end{remark}

\begin{remark}\label{Rigid-algebras}
	Another type of algebras where the multiplication $\mu \colon A \otimes A \to A$ is assumed to admit a right adjoint are \emph{(locally) rigid algebras} \cite[Chapter 1, \S 9]{GaitsgoryRozenblyum2017StudyDerivedAlgebraica}, \cite{Ramzi2024DualizablePresentableinftycategories}, \cite{Neuhauser2026RigidAlgebrasCospans}: locally rigid algebras are algebras in a monoidal $(\infty,2)$-category $\sX$ such that
	\begin{itemize}
		\item $\mu \colon A \otimes A \to A$ is left adjoint as an $(A,A)$-bimodule morphism, and
		\item $A$ is dualizable in $\sX$.
	\end{itemize}
	This definition is quite orthogonal to that of a lax-idempotent algebras:
	namely, if $A$ is lax-idempotent and satisfies the first property stated above, then it is in fact idempotent, meaning that $\mu \colon A \otimes A \to A$ is an equivalence.
	To see this, suppose $A$ is lax-idempotent and that $\mu$ admits a right adjoint in $\mathrm{Bimod}_{(A,A)}(\sX)$.
	Since $u \otimes \id \colon A \to A \otimes A$ is right adjoint to $\mu$ in $\sX$, it follows that $u \otimes \id$ is a left $A$-module morphism.
	This implies that the square
	\[\begin{tikzcd}[ampersand replacement=\&, column sep = 5em]
		{A \otimes A} \& {A \otimes A \otimes A} \\
		A \& {A \otimes A}
		\arrow["{\id_A \otimes (u \otimes \id_A)}", from=1-1, to=1-2]
		\arrow["\mu"', from=1-1, to=2-1]
		\arrow["{\mu \otimes \id}", from=1-2, to=2-2]
		\arrow["{u \otimes \id_A}"', from=2-1, to=2-2]
	\end{tikzcd}\]
	commutes.
	Since $(\mu \otimes \id_A) \circ (\id_A \otimes u \otimes \id_{A}) \simeq \id_{A \otimes A}$, it follows that $u \otimes \id_A$ is left inverse to $\mu$.
	On the other hand, $u \otimes \id_A$ is right inverse to $\mu$ by definition of an algebra, hence $\mu$ must be invertible.
\end{remark}

Finally, let us include the following observation due to Phil Pützstück that lax-idempotent algebras are a phenomenon that is specific to the noncommutative setting.

\begin{proposition}\label{prop:Phil}
	Let $\sX$ be a monoidal $(\infty,2)$-category and suppose that there exists a natural equivalence
	\[\beta_{x,y} \colon x \otimes y \eqarrow y \otimes x\]
	such that for any $x$ in $\sX$, the diagram
	\begin{equation}\label{eq:unit-braid-triangle}
	\begin{tikzcd}[ampersand replacement=\&]
		\& x \& \\
		{\unit \otimes x} \& {x \otimes \unit} \& {\unit \otimes x}
		\arrow["\sim"', from=1-2, to=2-1]
		\arrow["\sim"', from=1-2, to=2-2]
		\arrow["\sim", from=1-2, to=2-3]
		\arrow["\sim","{\beta_{\unit,x}}"', from=2-1, to=2-2]
		\arrow["\sim","{\beta_{x,\unit}}"', from=2-2, to=2-3]
	\end{tikzcd}
	\end{equation}
	commutes.
	Then any lax-idempotent algebra in $\sX$ is idempotent.
	In particular, this holds if $\sX$ is $E_n$-monoidal for some $n \geq 2$.
\end{proposition}

\begin{proof}
	Let $A$ be a lax-idempotent algebra.
	By naturality of $\beta$ and commutativity of \cref{eq:unit-braid-triangle}, it follows that
	\[\begin{tikzcd}[ampersand replacement=\&]
		\& {A \otimes A} \& \\
		{A \otimes A \otimes A} \& {A \otimes A \otimes A} \& {A \otimes A \otimes A}
		\arrow["{u \otimes \id_A \otimes \id_A}"', from=1-2, to=2-1]
		\arrow["{\id_A \otimes u \otimes \id_A}"{description}, from=1-2, to=2-2]
		\arrow["{u \otimes \id_A \otimes \id_A}", from=1-2, to=2-3]
		\arrow["{\beta_{A , A} \otimes \id_A}"', from=2-1, to=2-2]
		\arrow["{\beta_{A , A} \otimes \id_A}"', from=2-2, to=2-3]
	\end{tikzcd}\]
	commutes.
	Since $\mu$ is left adjoint to $u \otimes \id_A$, taking left adjoints in the left-hand triangle shows that the bottom triangle of the diagram
	\[\begin{tikzcd}[ampersand replacement=\&]
		{A \otimes A} \&\& A \\
		\& {A \otimes A \otimes A} \\
		{A \otimes A \otimes A} \&\& {A \otimes A}
		\arrow["\mu", from=1-1, to=1-3]
		\arrow["{u \otimes \id_A \otimes \id_A}", from=1-1, to=2-2]
		\arrow["{\id_A \otimes u \otimes \id_A}"', from=1-1, to=3-1]
		\arrow["{u \otimes \id_A}", from=1-3, to=3-3]
		\arrow["\sim", "\beta_{A,A} \otimes \id_A"'{pos=0.60}, from=3-1, to=2-2]
		\arrow["{\id_A \otimes \mu}", from=2-2, to=3-3]
		\arrow["{\mu \otimes \id_A}"', from=3-1, to=3-3]
	\end{tikzcd}\]
	commutes.
	The left-hand triangle of this diagram commutes by the right-hand triangle of the first diagram, while the rest of the diagram commutes by functoriality of $\otimes$.
	It follows that the outer rectangle commutes, and hence that
	\[(u \otimes \id_A) \circ \mu \simeq (\mu \otimes \id_A) \circ (\id_A \otimes u \otimes \id_A) \simeq \id_{A \otimes A}\]
	by unitality of $A$.
	We conclude that $u \otimes \id_A$ is a two-sided inverse of $\mu \colon A \otimes A \to A$, hence $A$ is idempotent.

	For the final sentence, observe that the braiding $\beta$ of any $E_n$-monoidal structure is a natural equivalence such that \cref{eq:unit-braid-triangle} commutes.
\end{proof}

\section{Monads on \texorpdfstring{$(\infty,2)$}{(∞,2)}-categories}\label{sec:heine}

The $(\infty,1)$-category $\Cat_{(\infty,2)}$ of $(\infty,2)$-categories is cartesian closed by \cite[Corollary 4.9]{Nuiten2024SuperStraight} or \cite{Rezk2010Cartesian}, with internal hom $\Fun(-,-)$.
Given an $(\infty,2)$-category $\sX$, Lurie's theory of endomorphism objects \cite[\S 4.7.1]{HA} endows the $(\infty,2)$-category $\End(\sX) = (\Fun(\sX,\sX),\circ)$ with a monoidal structure together with a left action $\Fun(\sX,\sX) \curvearrowright \sX$ given by composition.

\begin{definition}[Monads]
	Let $\sX$ be an $(\infty,2)$-category. A \emph{monad} on $\sX$ is an algebra object in the monoidal $(\infty,2)$-category $\Fun(\sX,\sX)$.
	The $(\infty,2)$-category $\Alg(\Fun(\sX,\sX))$ of monads on $\sX$ will be denoted $\Monad(\sX)$.
\end{definition}

We will write $(-)^{\leq 1}$ for the functor sending an $(\infty,2)$-category to its underlying $(\infty,1)$-category by forgetting all non-invertible 2-morphisms.

\begin{construction}[The $(\infty,2)$-category of algebras over a monad]\label{const:2talg}
	Given a monad $T$ on an $(\infty,2)$-category $\sX$, we obtain a monad $T^{\leq 1}$ on the $(\infty,1)$-category $\sX^{\leq 1}$.
	Taking $\cV$ to be $\Cat_{(\infty,1)}$ in \cite[Corollary 5.41.(4)]{Heine2026DualityMonadsMonadic}, we obtain a $\Cat_{(\infty,1)}$-enrichment on $\Alg_{T^{\leq 1}}(\sX^{\leq 1}) \coloneqq \LMod_{T^{\leq 1}}(\sX^{\leq 1})$.
	We define the \emph{$(\infty,2)$-category of $T$-algebras} $\Alg_T(\sX)$ to be the corresponding $(\infty,2)$-category under the equivalence between $\Cat_{(\infty,1)}$-enriched $(\infty,1)$-categories and $(\infty,2)$-categories \cite[Corollary 7.21]{Haugseng2015RectificationEnrichedCategories}.
\end{construction}

The following now holds by construction.

\begin{lemma}
	The underlying $(\infty,1)$-category of $\Alg_T(\sX)$ is $\Alg_{T^{\leq 1}}(\sX^{\leq 1})$. \qed
\end{lemma}

Recall that for an $(\infty,1)$-category $\cC$, there is an equivalence between (the opposite of) the $(\infty,1)$-categories of monads on $\cC$ and of monadic right adjoints to $\cC$.
The following extension to $(\infty,2)$-categories follows from \cite[Theorem 4.33]{Heine2026DualityMonadsMonadic}.

\begin{proposition}\label{prop:monadic-Heine-adj-2}
	Let $\sX$ be an $(\infty,2)$-category.
	Then there is a reflective adjunction
	\[\begin{tikzcd}[ampersand replacement=\&]
		{\mathrm{E} : (\Cat_{(\infty,2)}^{\leq 1})_{/\sX}^\mathrm{R}} \&\& {\Monad(\sX)^{\leq 1,\op} : \Alg_{(-)}(\sX)}
		\arrow[""{name=0, anchor=center, inner sep=0}, shift left=2, from=1-1, to=1-3]
		\arrow[""{name=1, anchor=center, inner sep=0}, shift left=2, hook', from=1-3, to=1-1]
		\arrow["\dashv"{anchor=center, rotate=-90}, draw=none, from=0, to=1]
	\end{tikzcd}\]
	of $(\infty,1)$-categories, where $(\Cat_{(\infty,2)}^{\leq 1})_{/\sX}^\mathrm{R}$ denotes the full sub-$(\infty,1)$-category of $(\Cat_{(\infty,2)}^{\leq 1})_{/\sX}$ spanned by the right adjoints.
	Here $\Alg_{(-)}(\sX)$ sends a monad $T$ to the $(\infty,2)$-category $\Alg_T(\sX)$ defined in \cref{const:2talg}, and $\mathrm{E}$ sends a right adjoint to its associated monad.
\end{proposition}

\begin{proof}
	It follows from Corollaries 5.39 and 5.41 of \cite{Heine2026DualityMonadsMonadic} that, for any $(\infty,2)$-category $\sX$ and any monad $T$ on $\sX$, the $(\infty,2)$-category $\Alg_T(\sX)$ from \cref{const:2talg} is an Eilenberg--Moore object.
	By \cite[Theorem 4.33]{Heine2026DualityMonadsMonadic}, we obtain a reflective adjunction
	\[\begin{tikzcd}[ampersand replacement=\&]
		{\mathrm{End} : \Fun([1],\Cat_{(\infty,2)}^{\leq 1})^\mathrm{R}} \&\& {\mathrm{Mon}(\Cat_{(\infty,2)})^{\op} : \Alg}
		\arrow[""{name=0, anchor=center, inner sep=0}, shift left=2, from=1-1, to=1-3]
		\arrow[""{name=1, anchor=center, inner sep=0}, shift left=2, hook', from=1-3, to=1-1]
		\arrow["\dashv"{anchor=center, rotate=-90}, draw=none, from=0, to=1]
	\end{tikzcd}\]
	of $(\infty,1)$-categories.
	Here $\mathrm{Mon}(\Cat_{(\infty,2)})$ denotes the $(\infty,1)$-category of monads defined in \cite[Definition 4.26]{Heine2026DualityMonadsMonadic}; observe that $\mathrm{Mon}_{\sX}(\Cat_{(\infty,2)})$ in the sense of \cite[Notation 4.27]{Heine2026DualityMonadsMonadic} is equivalent to $\Monad(\sX)^{\leq 1}$ in our notation.
	An inspection of the proof of \cite[Theorem 4.33]{Heine2026DualityMonadsMonadic} shows that this adjunction is fibered over $\Cat_{(\infty,2)}^{\leq 1}$ in the sense of \cite[Definition 7.3.2.2]{HA}.
	The adjunction from the statement is now obtained by taking fibers over $\sX \in \Cat_{(\infty,2)}$.
\end{proof}

Let us briefly unwind what the left adjoint $\mathrm{E}$ does to morphisms.
Let
\[\begin{tikzcd}[sep = {3em,between origins}]
	\sY && \sZ \\
	& \sX
	\arrow["F", from=1-1, to=1-3]
	\arrow["{U_1}"', from=1-1, to=2-2]
	\arrow["{U_2}", from=1-3, to=2-2]
\end{tikzcd}\]
be a commutative triangle in $\Cat_{(\infty,2)}$, where $U_1$ and $U_2$ admit left adjoints $T_1$ and $T_2$.
Unwinding the proof of \cite[Theorem 4.33]{Heine2026DualityMonadsMonadic} shows that the natural transformation $\mathrm{E}(F) \colon \mathrm{E}(U_2) = U_2T_2 \Rightarrow U_1 T_1 = \mathrm{E}(U_1)$
is given by the pasting
\[\begin{tikzcd}[sep = {3em,between origins}]
	& \sX & \\
	\sY && \sZ \\
	& \sX. &
	\arrow["{T_1}"', from=1-2, to=2-1]
	\arrow[""{name=0, anchor=center, inner sep=0}, "{T_2}", from=1-2, to=2-3]
	\arrow["F"', from=2-1, to=2-3]
	\arrow["{U_1}"', from=2-1, to=3-2]
	\arrow["{U_2}", from=2-3, to=3-2]
	\arrow[shorten <= 6, shorten >= 12, Rightarrow, from=0, to=2-1]
\end{tikzcd}\]
where the top triangle is the mate of the bottom.
More precisely, $\mathrm{E}(F)$
is the composite
\begin{equation}\label{eq:EF}
U_2T_2 \xRightarrow{\eta_1} U_2T_2U_1T_1 \simeq U_2T_2U_2FT_1 \xRightarrow{\varepsilon_2} U_2FT_1 \simeq U_1T_1.
\end{equation}

Let us write
\[(\Cat_{(\infty,2)}^{\leq 1})_{/\sX}^{\mathrm{R},T} \coloneqq (\Cat_{(\infty,2)}^{\leq 1})_{/\sX}^{\mathrm{R}} \times_{\Monad(\sX)^{\leq 1, \op}} \{T\}\]
for the $(\infty,1)$-category of right adjoints $\sY \to \sX$ whose associated monad is $T$.
It follows directly from \cref{prop:monadic-Heine-adj-2} that $\Alg_T(\sX) \to \sX$ is terminal in this $(\infty,1)$-category.

\begin{corollary}\label{algT-terminal}
	Let $\sX$ be an $(\infty,2)$-category and let $T$ be a monad on $\sX$.
	Then the forgetful functor $\Alg_T(\sX) \to \sX$ is the terminal object in $(\Cat_{(\infty,2)}^{\leq 1})_{/\sX}^{\mathrm{R},T}$. \qed
\end{corollary}

In particular, for any adjunction $F : \sX \rightleftarrows \sY : U$ between $(\infty,2)$-categories, we obtain a canonical comparison functor $\sY \to \Alg_{UF}(\sX)$.

\begin{definition}[Monadic adjunctions]
	Let $F : \sX \rightleftarrows \sY : U$ be an adjunction between $(\infty,2)$-categories.
	We say that such an adjunction is \emph{monadic}, or that $U$ is \emph{monadic}, if the canonical functor $\sY \to \Alg_{UF}(\sX)$ is an equivalence of $(\infty,2)$-categories.
\end{definition}

Note that this agrees with Heine's definition of a monadic right adjoint \cite[Definition 3.26]{Heine2026DualityMonadsMonadic}.

\section{Lax-idempotent monads}\label{sec:limonad}

We now specialize the statements from \cref{sec:laxi-algebras} to the case of endofunctors on an $(\infty,2)$-category.
Note that \cref{deff,def:laxi-adjunctions} can be stated in an arbitrary $(\infty,3)$-category, not just $\Cat_{(\infty,2)}$.

\begin{definition}[Lax-idempotent monads]\label{deff}
	Let $\sX$ be an $(\infty,2)$-category.
	An \emph{(op)lax-idempotent (co)monad} $T$ on $\sX$ is an (op)lax-idempotent (co)algebra in $\Fun(\sX,\sX)$.
\end{definition}

\begin{remark}
	The definitions of lax-idempotent comonads and oplax-idempotent comonads are sometimes swapped in the classical $(2,2)$-categorical literature.
	Our convention agrees with the one of \cite{LackShulman2012Enhanced2categoriesLimits}, but it is opposite to the convention of \cite{NLABlimonad}.
\end{remark}

For brevity, we focus on the case of lax-idempotent monads.
By appropriately dualizing, one obtains analogous statements for (op)lax-idempotent (co)monads.
Whether a monad is lax-idempotent can be checked pointwise.

\begin{lemma}\label{lem:li-monad-ptwise}
	Let $T$ be a monad on an $(\infty,2)$-category $\sX$.
	Then $T$ is lax-idempotent if and only if for every object $x$ in $\sX$, the multiplication map $\mu_x \colon TTx \to Tx$ is left adjoint to the unit map $\eta_{Tx} \colon Tx \to TTx$.
\end{lemma}

\begin{proof}
	The ``only if'' direction is clear.
	For the converse, note that it follows from \cite[Corollary 4.2.12]{AbellanGagnaea2025StraighteningLaxTransformations} that $\eta T \colon T \to TT$ is a right adjoint in $\Fun(\sX,\sX)$ if for any $f \colon x \to y$ in $\sX$, the mate square
	\[\begin{tikzcd}[ampersand replacement=\&]
		TTx \& TTy \\
		Tx \& Ty
		\arrow["TTf", from=1-1, to=1-2]
		\arrow["{\mu_x}"', from=1-1, to=2-1]
		\arrow[between={0.2}{0.8}, Rightarrow, from=1-2, to=2-1]
		\arrow["{\mu_y}", from=1-2, to=2-2]
		\arrow["Tf"', from=2-1, to=2-2]
	\end{tikzcd}\]
	commutes.
	By \cref{BCeasy}, this is the case precisely if there exists some homotopy making the square commute, which follows from the naturality of $\mu$.
	We thus obtain a left adjoint $\overline{\mu}\colon TT \to T$ from \cite[Corollary 4.2.12]{AbellanGagnaea2025StraighteningLaxTransformations}, and what remains is to show that $\overline{\mu} \simeq \mu$. Whiskering the unit $\id \Rrightarrow \eta T \circ \overline{\mu}$ with
	$\mu$ one obtains a morphism $\mu \Rrightarrow \overline{\mu}$ which we claim is an equivalence.
	To see this, note that this can be verified for each $x \in \sX$ separately, where it follows directly from our hypothesis.
\end{proof}

We similarly obtain the following analogue of \cref{altAlg}.

\begin{proposition}\label{alt}
	Let $T$ be a monad on an $(\infty,2)$-category $\sX$. Then the following are equivalent:
	\begin{enumerate}[(1)]
		\item\label{alt1} $T$ is lax-idempotent.
		\item\label{alt2} For any $x$ in $\sX$, the multiplication $\mu_x \colon TTx \to Tx$ is right adjoint to $T(\eta_x)$.
		\item\label{alt3} There exists a modification $d \colon T\eta \Rrightarrow \eta T$ such that for every $x$ in $\sX$, the 2-morphisms $d_x \whisk \eta_x \colon T(\eta_x) \eta_x \Rightarrow \eta_{Tx} \eta_x$ and $\mu_x \whisk d_x \colon \mu_x T(\eta_x) \Rightarrow \mu_x \eta_{Tx}$ are invertible in $\sX$. \qed
	\end{enumerate}
\end{proposition}

One can also define \emph{lax-idempotent adjunctions}, which relate to lax-idempotent monads in the same way that idempotent adjunctions relate to idempotent monads.

\begin{definition}[Lax-idempotent adjunctions]\label{def:laxi-adjunctions}
	An adjunction $F : \sX \rightleftarrows \sY : G$ between $(\infty,2)$-categories is called \emph{lax-idempotent} if the corresponding monad $GF$ on $\sX$ is lax-idempotent.
\end{definition}

\begin{proposition}\label{laxi-adj-conditions}
	Let $F : \sX \rightleftarrows \sY : G$ be an adjunction between $(\infty,2)$-categories with unit $\eta \colon \id_{\sX} \Rightarrow GF$ and counit $\varepsilon \colon FG \Rightarrow \id_{\sY}$. The following are equivalent:
	\begin{enumerate}[(1)]
		\item\label{laxi-adj1} This adjunction is lax-idempotent, i.e.\ the monad $GF$ is lax-idempotent.
		\item\label{laxi-adj2} The comonad $FG$ is lax-idempotent.
		\item\label{laxi-adj3} For any $x$ in $\sX$, one has $F\eta_x \dashv \varepsilon_{Fx}$.
		\item\label{laxi-adj4} For any $y$ in $\sY$, one has $G\varepsilon_y \dashv \eta_{Gy}$.
	\end{enumerate}
\end{proposition}

\begin{proof}
	Write $T=GF$ for brevity.
	Item \cref{laxi-adj3} implies $T\eta_x \dashv G \varepsilon_{Fx} = \mu_x$, which implies \cref{laxi-adj1} by \cref{alt}.
	We now prove \cref{laxi-adj4} from \cref{laxi-adj1}.
	Let $d_x \colon T\eta_x \Rightarrow \eta_{Tx}$ be as in item \cref{alt3} of \cref{alt}.
	Observe that $G(\varepsilon_y) \eta_{Gy} \simeq \id_{Gy}$ by the triangle identities.
	Form the composite
	\[\nu \colon \id_{TG y} \simeq TG(\varepsilon_y)T(\eta_{Gy}) \xRightarrow{TG(\varepsilon_y) \whisk d_{Gy}} TG(\varepsilon_y)\eta_{TGy} \simeq \eta_{Gy} G(\varepsilon_y),\]
	where we use the naturality of $\eta$ for the last equivalence.
	It follows directly from \cref{alt} that $\nu \whisk \eta_{Gy}$ is invertible, while $G(\varepsilon_y) \whisk \nu$ is invertible since $G(\varepsilon_y) \whisk (TG(\varepsilon_y) \whisk d_{Gy}) \simeq (G(\varepsilon_y) G(\varepsilon_{FGy})) \whisk d_{Gy}$ and $G(\varepsilon_{FGy}) \colon TT(Gy) \to T(Gy)$ is the monad multiplication map.
	It follows by \cref{easyadjun} that $G(\varepsilon_y) \dashv \eta_{Gy}$.
	The dual arguments now show that \cref{laxi-adj4} implies \cref{laxi-adj2} implies \cref{laxi-adj3}, completing the proof.
\end{proof}

\begin{remark}
	In the setting of ordinary $(2,2)$-categories, many more equivalent definitions of lax-idempotent adjunctions are discussed in \cite{NLABliadj}.
	All these definitions are also equivalent in the setting of $(\infty,2)$-categories, which follows by essentially the same arguments as in \cite{NLABliadj}.
\end{remark}

\section{Algebras over a lax-idempotent monad}\label{sec:talg}

The main result of this paper is an explicit description of $\Alg_T(\sX)$ as a certain locally full sub-$(\infty,2)$-category $\sX^T$ of $\sX$.
Recall that a \emph{pointed endofunctor} is an endofunctor $T$ equipped with a natural transformation $\id \Rightarrow T$.
Any monad $T$ can be viewed as a pointed endofunctor via the unit transformation $\eta \colon \id \Rightarrow T$.

\begin{construction}[The $(\infty,2)$-category $\sX^T$]\label{talg}
	Let $\eta \colon \id_{\sX} \Rightarrow T$ be a pointed endofunctor on an $(\infty,2)$-category $\sX$.
	Define $\sX^T$ to be the following locally full sub-$(\infty,2)$-category of $\sX$:
	\begin{enumerate}[(1)]
		\item Its objects are those $x$ for which $\eta_x \colon x \to Tx$ admits a reflective left adjoint $\alpha_x$; that is, a left adjoint $\alpha_x$ such that $\alpha_x \circ \eta_x$ is an equivalence.
		\item\label{talg:item2} Its 1-morphisms are those $f \colon x \to y$ for which there exists a homotopy making the square
		\[\begin{tikzcd}
			Tx & Ty \\
			x & y
			\arrow["Tf", from=1-1, to=1-2]
			\arrow["{\alpha_x}"', from=1-1, to=2-1]
			\arrow["{\alpha_y}", from=1-2, to=2-2]
			\arrow["f"', from=2-1, to=2-2]
		\end{tikzcd}\]
		commute.
	\end{enumerate}
	Given an object $x$ in $\sX^T$, we will always write $\alpha_x$ for the left adjoint of $\eta_x \colon x \to Tx$.
\end{construction}

\begin{remark}\label{rem:BCeasy-XT}
	By \cref{BCeasy}, the second condition is equivalent to the Beck--Chevalley transformation
	\[\begin{tikzcd}
			Tx & Ty \\
			x & y
			\arrow["Tf", from=1-1, to=1-2]
			\arrow["{\alpha_x}"', from=1-1, to=2-1]
			\arrow[shorten <= 6, shorten >= 6, Rightarrow, from=1-2, to=2-1]
			\arrow["{\alpha_y}", from=1-2, to=2-2]
			\arrow["f"', from=2-1, to=2-2]
		\end{tikzcd}\]
	being invertible.
\end{remark}

\begin{remark}\label{rem:BCeasy-FF}
	By \cref{lem:lali}, the condition on objects is equivalent to $\eta_x$ being a fully faithful right adjoint.
	For many lax-idempotent monads $T$ encountered in nature, all unit morphisms $\eta_x \colon x \to Tx$ are fully faithful.
	In this case an object lies in $\sX^T$ if and only if $\eta_x$ admits a left adjoint.

	For completeness' sake, let us give an example of a lax-idempotent monad where not all unit morphisms $\eta_x \colon x \to Tx$ are fully faithful: consider the constant functor $*$ on $\Cat_{(\infty,1)}$.
	This is clearly a lax-idempotent monad.
	The unit $\eta_\cC \colon \cC \to *$ is right adjoint precisely if $\cC$ admits an initial object, whereas it is a fully faithful right adjoint if and only if $\cC \simeq *$.
\end{remark}

We make the following observation about the left adjoint $\alpha_x$:

\begin{lemma}\label{alphaTalpha}
	Let $T$ be a lax-idempotent monad on an $(\infty,2)$-category $\sX$ and let $x$ be an object of $\sX^T$. Then $\alpha_x T(\alpha_x) \simeq \alpha_x \mu_x$.
\end{lemma}

\begin{proof}
	These morphisms are left adjoint to $T(\eta_{x})\eta_x \simeq \eta_{Tx} \eta_x $ by assumption.
\end{proof}

The rest of this section is devoted to proving that $\Alg_T(\sX) \simeq \sX^T$.
We will prove this by constructing an equivalence $\Alg_T(\sX) \to \sX^T$ over $\sX$.
The proof uses the following existence result for left Kan extensions along $\eta_x \colon x \to Tx$:

\begin{lemma}\label{lkan-exist}
	Let $T$ be a lax-idempotent monad on an $(\infty,2)$-category $\sX$, let $x$ and $y$ be objects in $\sX$ and suppose that $y$ lies in $\sX^T$.
	Then the composite
	\[\twomap_{\sX}(x,y) \xrightarrow{T} \twomap_{\sX}(Tx,Ty) \xrightarrow{(\alpha_y)_*} \twomap_{\sX}(Tx,y)\]
	is a fully faithful left adjoint of
	$\eta_x^* \colon \twomap_{\sX}(Tx,y) \to \twomap_{\sX}(x,y)$.
\end{lemma}

\begin{proof}
	Let $g \colon x \to y$ be given.
	Using that $\alpha_y \eta_y \simeq \id_y$ by \cref{lem:lali}, we have natural equivalences
	\[\alpha_y T(g) \eta_x \simeq \alpha_y \eta_y g \simeq g,\]
	hence we obtain a natural equivalence $\eta_x^* \circ (\alpha_y)_* T \simeq \id$.
	We will now construct a natural transformation $\nu \colon (\alpha_y)_* T \circ \eta^*_x \Rightarrow \id$. Let $d_x \colon T(\eta_{x}) \Rightarrow \eta_{Tx}$ be the 2-morphism from item \cref{alt3} of \cref{alt}. For any $f \colon Tx \to y$, we can form the map
	\[\alpha_y T(f \eta_x) \simeq \alpha_y T(f) T(\eta_x) \xRightarrow{d_x} \alpha_y T(f) \eta_{Tx} \simeq \alpha_y \eta_y f \simeq f.\]
	This map is natural in $f$ and hence defines a transformation $\nu \colon (\alpha_y)_* T \eta^*_x \Rightarrow \id$.

	By (the dual of) \cref{easyadjun}, the result follows if we can show that $\eta_x^* \cdot \nu$ and $\nu \cdot ((\alpha_y)_*T)$ are invertible.
	Since the case of $\eta_x^* \cdot \nu$ is similar but easier, we will only show this for $\nu \cdot ((\alpha_y)_*T)$.
	This comes down to showing that
	\[\alpha_y T(\alpha_yT(g)) \circ T\eta_{x} \xRightarrow{(\alpha_y T(\alpha_yT(g))) \whisk d_{x}} \alpha_y T(\alpha_y T(g)) \circ \eta_{Tx}\]
	is invertible for any $g \colon x \to y$.
	This follows since
	\[\alpha_y T(\alpha_y T(g)) \simeq \alpha_y T(\alpha_y) TT(g) \simeq \alpha_y \mu_y TT(g) \simeq \alpha_y T(g) \mu_{x}\]
	by \cref{alphaTalpha}, and $\mu_{x} \whisk d_{x}$ is invertible by \cref{alt}.
\end{proof}

We use this to show that $T$ defines a left adjoint to the inclusion $\sX^T \hookrightarrow \sX$.

\begin{lemma}\label{step1}
	Let $T$ be a lax-idempotent monad on an $(\infty,2)$-category $\sX$. Then $T \colon \sX \to \sX$ lands in the sub-$(\infty,2)$-category $\sX^T$ and is left adjoint to the inclusion $\sX^T \hookrightarrow \sX$.
\end{lemma}

\begin{proof}
	For any $x$ in $\sX$, the map $\eta_{Tx} \colon Tx \to TTx$ admits a reflective left adjoint $\mu_x \colon TTx \to Tx$ by \cref{lem:li-monad-ptwise}, hence $Tx$ lies in $\sX^T$ for any $x$ in $\sX$.
	Moreover, if $f \colon x \to y$ is a 1-morphism in $\sX$, then $Tf \colon Tx \to Ty$ lies in $\sX^T$ by naturality of $\mu$, hence $T \colon \sX \to \sX$ lands in $\sX^T$.

	The unit $\eta \colon \id_\sX \Rightarrow T$ induces a natural map
	\[\twomap_{\sX^T}(Tx,y) \xhookrightarrow{\iota} \twomap_{\sX}(Tx,y) \xrightarrow{\eta_x^*} \twomap_{\sX}(x,y).\]
	The result follows if this composite is an equivalence.
	By \cref{lkan-exist}, it suffices to show that the essential image of $\iota$ agrees with the essential image of
	\[\twomap_{\sX}(x,y) \xrightarrow{T} \twomap_{\sX}(Tx,Ty) \xrightarrow{\alpha_y \circ -} \twomap_{\sX}(Tx,y);\]
	in other words, we need to show that $f \colon Tx \to y$ lies in $\sX^T$ if and only if $f$ is of the form $\alpha_y \circ T(f')$ for some $f' \colon x \to y$.

	For the ``only if'' direction, take $f' = f \circ \eta_x$.
	For the ``if'' direction, observe that $\alpha_y T(f') \mu_x \simeq \alpha_y \mu_y TT(f') \simeq \alpha_y T(\alpha_y Tf')$ by \cref{alphaTalpha}.
\end{proof}

\begin{lemma}\label{step2}
Let $F : \sX \rightleftarrows \sY : G$ be an adjunction between $(\infty,2)$-categories such that the monad $T \coloneqq GF$ on $\sX$ is lax-idempotent.
Then $G$ lands in the sub-$(\infty,2)$-category $\sX^T$ of $\sX$.
\end{lemma}

\begin{proof}
	It follows from \cref{laxi-adj-conditions} that $G(\varepsilon_y) \dashv \eta_{Gy}$, hence $Gy$ lies in $\sX^T$ for any $y$ in $\sY$.
	Moreover, for any map $f \colon y \to z$ in $\sY$ we have $G(\varepsilon_z) TG(f) \simeq G(f) G(\varepsilon_y)$ by naturality of $\varepsilon$, hence $Gf$ lies in $\sX^T$.
\end{proof}

We are now ready to establish the equivalence $\Alg_T(\sX) \simeq \sX^T$, proving \cref{thmA}.

\begin{theorem}\label{main-thm-1}
	Let $T$ be a lax-idempotent monad on an $(\infty,2)$-category $\sX$.
	Then there is an equivalence $H \colon \Alg_T(\sX) \eqarrow \sX^T$ such that the triangle
	\[\begin{tikzcd}[column sep={3em,between origins}]
		{\Alg_T(\sX)} && {\sX^T} \\
		& \sX
		\arrow["\sim","H"', from=1-1, to=1-3]
		\arrow["U"', from=1-1, to=2-2]
		\arrow[hook, from=1-3, to=2-2]
	\end{tikzcd}\]
	commutes.
\end{theorem}

\begin{proof}
	By \cref{step2}, the forgetful functor $U$ lands in $\sX^T \subset \sX$, giving us the functor $H \colon \Alg_T(\sX) \to \sX^T$ from the statement.
	To see that it is an equivalence, write $T'$ for the monad associated to the right adjoint $\sX^T \hookrightarrow \sX$.
	It follows from the proof of \cref{step1} that the underlying endofunctor of $T'$ is $T$, and that the unit map of the monad $T'$ agrees with the unit map $\eta \colon \id \Rightarrow T$.
	By \cref{prop:monadic-Heine-adj-2}, we obtain a map of monads $T' \Rightarrow T$, and it follows from the description of this map given in \cref{eq:EF} that it is an equivalence of endofunctors on $\sX$, and hence of monads.
	In particular, we may view $\sX^T \hookrightarrow \sX$ as an object of $(\Cat_{(\infty,2)}^{\leq 1})^{\mathrm{R},T}_{/\sX}$.
	Since $\Alg_T(\sX)$ is the terminal object of this $(\infty,1)$-category by \cref{algT-terminal} and $\sX^T \hookrightarrow \sX$ is a monomorphism, it follows that the functor $\Alg_T(\sX) \to \sX^T$ must be an equivalence of $(\infty,2)$-categories.
\end{proof}

By reversing the 1- and/or 2-morphisms in $\sX$, we obtain analogous results for (op)lax-idempotent (co)monads.
Let us spell these out for completeness.

\begin{theorem}\label{main-thm-2}
	Let $\sX$ be an $(\infty,2)$-category and $T$ an oplax-idempotent monad on $\sX$.
	Then $\Alg_T(\sX) \to \sX$ is a locally full sub-$(\infty,2)$-category inclusion
	\begin{enumerate}
		\item whose objects are those $x$ for which $\eta_x \colon x \to Tx$ admits a coreflective right adjoint $\alpha_x$, and
		\item whose 1-morphisms are those $f \colon x \to y$ for which $f \alpha_x \simeq \alpha_y T(f)$.
	\end{enumerate}
	If $T$ is a lax-idempotent (resp.\ oplax-idempotent) comonad on $\sX$, then $\coAlg_T(\sX) \hookrightarrow \sX$ is a locally full sub-$(\infty,2)$-category inclusion
		\begin{enumerate}
		\item whose objects are those $x$ for which $\varepsilon_x \colon Tx \to x$ admits a fully faithful left (resp.\ right) adjoint $\nu_x$, and
		\item whose 1-morphisms are those $f \colon x \to y$ for which $\nu_y f \simeq T(f) \nu_x$.
	\end{enumerate}
\end{theorem}

Let us conclude this section with a converse to the main result.
Recall that the $(\infty,2)$-category $\sX^T$ from \cref{talg} makes sense for any monad $T$ on an $(\infty,2)$-category $\sX$, not just lax-idempotent ones.

\begin{proposition}\label{main-thm-inverse-2}
	An adjunction $F : \sX \rightleftarrows \sY : G$ between $(\infty,2)$-categories is monadic and lax-idempotent if and only if the following conditions hold:
	\begin{enumerate}[(1)]
		\item $G \colon \sY \to \sX$ lands in the sub-$(\infty,2)$-category $\sX^T$ of $\sX$, where $T$ is the monad $GF$.
		\item The induced functor $G \colon \sY \to \sX^T$ is an equivalence of $(\infty,2)$-categories.
	\end{enumerate}
\end{proposition}

\begin{proof}
	The ``only if'' direction is \cref{main-thm-1}.
	For the converse, note that monadicity follows from \cref{main-thm-1} once the adjunction is shown to be lax-idempotent.
	To this end, let $x$ in $\sX$ be given.
	Then $Tx = G(Fx)$ lies in $\sX^T$ by assumption, hence $\eta_{Tx}$ admits a left adjoint $\alpha_{Tx}$ such that $\alpha_{Tx}\eta_{Tx} \simeq \id_{Tx}$.
	The result therefore follows from \cref{lem:li-monad-ptwise} if we can show for any $x$ in $\sX$ that $\mu_x = G(\varepsilon_{Fx}) \colon TTx \to Tx$ is equivalent to $\alpha_{Tx}$.

	To see this, observe that $\alpha_{Tx}$ lies in $\sX^T$ since the squares
	\[\begin{tikzcd}[ampersand replacement=\&]
		TTTx \& TTx \\
		TTx \& Tx
		\arrow["{T(\alpha_{Tx})}", from=1-1, to=1-2]
		\arrow["{\alpha_{TTx}}"', from=1-1, to=2-1]
		\arrow[shorten <= 6, shorten >= 6, Rightarrow, from=1-2, to=2-1]
		\arrow["{\alpha_{Tx}}", from=1-2, to=2-2]
		\arrow["{\alpha_{Tx}}"', from=2-1, to=2-2]
	\end{tikzcd} \quad \text{and} \quad
	\begin{tikzcd}[ampersand replacement=\&]
		TTTx \& TTx \\
		TTx \& Tx
		\arrow["{T(\eta_{Tx})}"', from=1-2, to=1-1]
		\arrow["{\eta_{TTx}}", from=2-1, to=1-1]
		\arrow["\circlearrowright"{description}, draw=none, from=2-2, to=1-1]
		\arrow["{\eta_{Tx}}"', from=2-2, to=1-2]
		\arrow["{\eta_{Tx}}", from=2-2, to=2-1]
	\end{tikzcd}\]
	are adjoint to each other, hence the first one commutes.
	Because of the equivalence $G \colon \sY \eqarrow \sX^T$, we may lift $\alpha_{Tx}$ to a map $\alpha' \colon FTx \to Fx$ in $\sY$.
	Observe that the transpose of this map under the adjunction $F \dashv G$ is the composite
	\[Tx \xrightarrow{\eta_{Tx}} GFTx \xrightarrow{G(\alpha') \simeq \alpha_{Tx}} Tx,\]
	which is equivalent to $\id_{Tx}$ by definition of $\sX^T$.
	Since the transpose of $\varepsilon_{Fx} \colon FTx \to Fx$ is equivalent to $\id_{Tx}$ by definition, it follows that $\varepsilon_{Fx} \simeq \alpha'$ and hence that $\mu_{x} \simeq G(\varepsilon_{Fx}) \simeq G(\alpha') \simeq \alpha_{Tx}$ is left adjoint to $\eta_{Tx}$.
\end{proof}

\section{An adjoint functor theorem}\label{sec:adjfunthm}

We will now show that for any lax-idempotent monad $T$, one obtains an ``adjoint functor theorem''.
This result is inspired by \cite{ArkorLibertiea2024AdjointFunctorTheorems}, specifically Corollary 2.17.

\begin{theorem}\label{adjfunthm}
	Let $T$ be a lax-idempotent monad on an $(\infty,2)$-category $\sX$ and
	let $x$ and $y$ be objects of $\Alg_T(\sX) \subset \sX$.
	Then the following are equivalent for a morphism $f \colon x \to y$ in $\sX$:
	\begin{enumerate}[(1)]
		\item\label{adjfunthm-1} $f$ admits a right adjoint in $\sX$.
		\item\label{adjfunthm-2} $f$ lies in $\Alg_T(\sX) \subset \sX$ and $Tf \colon Tx \to Ty$ admits a right adjoint in $\sX$.
	\end{enumerate}
\end{theorem}

\begin{proof}
	Suppose $g$ is right adjoint to $f$.
	Then $Tg$ is right adjoint to $Tf$.
	Moreover, the square
	\[\begin{tikzcd}[ampersand replacement=\&]
		Tx \& Ty \\
		x \& y
		\arrow["Tg"', from=1-2, to=1-1]
		\arrow["{\eta_x}", from=2-1, to=1-1]
		\arrow["{\eta_y}", from=2-2, to=1-2]
		\arrow["g"', from=2-2, to=2-1]
	\end{tikzcd}\]
	commutes by naturality; taking left adjoints, it follows that $f$ lies in $\sX^T \simeq \Alg_T(\sX)$.
	We conclude that \cref{adjfunthm-1} implies \cref{adjfunthm-2}.

	For the converse, let $g'$ denote the right adjoint to $Tf$. By applying \cref{lem:lift-along-lali} to
	\[\begin{tikzcd}[ampersand replacement=\&]
		x \& Tx \& y
		\arrow["{\eta_x}"', shift right, hook, from=1-1, to=1-2]
		\arrow["{\alpha_x}"', shift right, from=1-2, to=1-1]
		\arrow[""{name=0, anchor=center, inner sep=0}, "{\alpha_y Tf}", from=1-2, to=1-3]
	\end{tikzcd}\]
	it follows that $f$ is left adjoint to $\alpha_x g' \eta_y$.
\end{proof}

\part{Examples}

\section{Cocartesian fibrations}\label{sec:cocart}

Given an $(\infty,1)$-category $\cB$, we write $\Cocart(\cB)$ for the locally full sub-$(\infty,2)$-category of $\Cat_{(\infty,1)/\cB}$ whose objects are cocartesian fibrations and whose 1-morphisms are the functors over $\cB$ that preserve cocartesian lifts.
By the dual of \cite[Theorem 5.2.8.(iii)]{RiehlVerity2022ElementsCategoryTheory}, $p \colon \cC \to \cB$ is a cocartesian fibration precisely if the inclusion $i \colon \cC \to \cC \times_{\cB} \Ar\cB$ admits a left adjoint $L_\cC$ relative to $\cB$:
\[\begin{tikzcd}[column sep=small]
	\cC && {\cC \times_{\cB}\Ar\cB} \\
	& \cB.
	\arrow[""{name=0, anchor=center, inner sep=0}, "i"', shift right=1, from=1-1, to=1-3]
	\arrow["p"',from=1-1, to=2-2]
	\arrow[""{name=1, anchor=center, inner sep=0}, shift right=0, bend right=30, from=1-3, to=1-1, dashed, "L_\cC"']
	\arrow["\ev_1",from=1-3, to=2-2]
	\arrow["\dashv"{anchor=center, rotate=-91}, draw=none, from=1, to=0]
\end{tikzcd}\]
The left adjoint $L_\cC$ takes a pair $(c,f)$ to the cocartesian transport $f_!c$.
This suggests the following result.

\begin{proposition}\label{prop:partial-cocart-monadic}
	Let $\cB$ be an $(\infty,1)$-category.
	Then the adjunction
	\[\begin{tikzcd}[ampersand replacement=\&]
		{- \times_\cB \Ar(\cB) : \Cat_{(\infty,1)/\cB}} \&\& {\Cocart(\cB) : \mathrm{incl}}
		\arrow[""{name=0, anchor=center, inner sep=0}, shift left=2, from=1-1, to=1-3]
		\arrow[""{name=1, anchor=center, inner sep=0}, hook', shift left=2, from=1-3, to=1-1]
		\arrow["\dashv"{anchor=center, rotate=-90}, draw=none, from=0, to=1]
	\end{tikzcd}\]
	between $(\infty,2)$-categories is monadic and lax-idempotent.
\end{proposition}

\begin{proof}
	Let us write $T$ for the monad associated to this adjunction.
	By \cref{main-thm-inverse-2}, it suffices to show that the locally full sub-$(\infty,2)$-category $(\Cat_{(\infty,1)/\cB})^T$ of $\Cat_{(\infty,1)/\cB}$, as defined in \cref{talg}, agrees with $\Cocart(\cB)$.
	Since the unit $\cC \hookrightarrow \cC \times_\cB \Ar(\cB)$ is always fully faithful, it follows from the dual of \cite[Theorem 5.2.8.(iii)]{RiehlVerity2022ElementsCategoryTheory} and \cref{rem:BCeasy-FF} that this is the case on objects.
	It therefore suffices to show that a functor $F \colon \cC \to \cD$ between cocartesian fibrations over $\cB$ preserves cocartesian morphisms if and only if the square
	\[\begin{tikzcd}[ampersand replacement=\&, column sep = large]
		{\cC \times_\cB \Ar(\cB)} \& {\cD \times_\cB \Ar(\cB)} \\
		\cC \& \cD
		\arrow["{F \times_\cB \Ar(\cB)}", from=1-1, to=1-2]
		\arrow["{L_\cC}", from=1-1, to=2-1]
		\arrow["{L_\cD}", from=1-2, to=2-2]
		\arrow["F"', from=2-1, to=2-2]
	\end{tikzcd}\]
	commutes.
	This is easily deduced from the fact that the left adjoints $L_\cC$ and $L_\cD$ take a pair $(x,f)$ to the cocartesian transport $f_!x$; we leave the details to the reader.
\end{proof}

\begin{remark}
	By taking opposites, we also obtain an analogue of \cref{prop:partial-cocart-monadic} for cartesian fibrations over $\cB$ where the adjunction is oplax-idempotent.
\end{remark}

\section{\texorpdfstring{$\infty$}{∞}-operads}\label{sec:operads}

Given an $\infty$-operad $\cP^\otimes \to \Fin_*$ in the sense of \cite[Definition 2.1.1.10]{HA}, its \emph{symmetric monoidal envelope} $\Env \cP$ has the universal property that for any symmetric monoidal $(\infty,1)$-category $\cC$, there is a natural equivalence
\[\Alg_\cP(\cC) = \twomap_{\Op_\infty}(\cP^\otimes,\cC^\otimes) \simeq \twomap_{\CMon(\Cat_{(\infty,1)})}(\Env\cP, \cC).\]
In other words, it is the left adjoint in an adjunction
\begin{equation}\label{OLD:Env-adjunction}
\begin{tikzcd}[ampersand replacement=\&]
	{\Env : \Op_\infty} \& {\CMon(\Cat_{(\infty,1)}) : \iota}
	\arrow[""{name=0, anchor=center, inner sep=0}, curve={height=-12pt}, from=1-1, to=1-2]
	\arrow[""{name=1, anchor=center, inner sep=0}, hook, from=1-2, to=1-1]
	\arrow["\dashv"{anchor=center, rotate=-90}, draw=none, from=0, to=1]
\end{tikzcd}
\end{equation}
of $(\infty,2)$-categories, where the right adjoint $\iota$ is the (non-full) inclusion of the $(\infty,2)$-category of symmetric monoidal $(\infty,1)$-categories into that of $\infty$-operads.

\begin{theorem}\label{OLD:Env-comonadic}
	The adjunction
	\[\begin{tikzcd}[ampersand replacement=\&]
		{\Env : \Op_\infty} \& {\CMon(\Cat_{(\infty,1)}) : \iota}
		\arrow[""{name=0, anchor=center, inner sep=0}, curve={height=-12pt}, from=1-1, to=1-2]
		\arrow[""{name=1, anchor=center, inner sep=0}, hook, from=1-2, to=1-1]
		\arrow["\dashv"{anchor=center, rotate=-90}, draw=none, from=0, to=1]
	\end{tikzcd}\]
	is lax-idempotent, monadic and comonadic.
\end{theorem}

In particular, \cref{main-thm-2} gives the following description of $\Op_\infty$ as a locally full sub-$(\infty,2)$-category of $\CMon(\Cat_{(\infty,1)})$.

\begin{corollary}\label{OLD:Env-faithful}
	The functor $\Env \colon \Op_{\infty} \to \CMon(\Cat_{(\infty,1)})$ is a locally fully faithful inclusion.
	Its image consists of those symmetric monoidal $(\infty,1)$-categories $\cC$ for which the counit $\varepsilon \colon \Env(\cC) \to \cC$ admits a left adjoint $\nu$, and those 1-morphisms $f \colon \cC \to \cD$ for which
	\[\begin{tikzcd}[ampersand replacement=\&]
		{\Env \cC} \& {\Env \cD} \\
		\cC \& \cD
		\arrow["{\Env f}", from=1-1, to=1-2]
		\arrow["\nu", from=2-1, to=1-1]
		\arrow["f"', from=2-1, to=2-2]
		\arrow["\nu", from=2-2, to=1-2]
	\end{tikzcd}\]
	commutes. \qed
\end{corollary}

The faithfulness of $\Env \colon \Op_\infty \to \CMon(\Cat_{(\infty,1)})$ was also shown in \cite[Theorem 3.2.13]{BarkanSteinebrunner2026EquifiberedApproachProperads} using different methods.
Note that our characterization of the image of $\Env \colon \Op_\infty \to \CMon(\Cat_{(\infty,1)})$ is different from the one given there.

We now turn to the proof of \cref{OLD:Env-comonadic}.
The main idea is that an $\infty$-operad $\cP$ can be recovered from $\Env(\cP)$ as the sub-$\infty$-operad spanned by the \emph{$\otimes$-primes}: these are the non-unit objects $c$ with the property that if $c \simeq d \otimes d'$, then either $d$ or $d'$ must be the unit.

\begin{proof}[Proof of \cref{OLD:Env-comonadic}]
	Given an $\infty$-operad $\cP^\otimes \to \Fin_*$, its envelope is the symmetric monoidal $(\infty,1)$-category defined by the cocartesian fibration $\ev_1 \colon \cP^\otimes \times_{\Fin_*} \Ar_\activ(\Fin_*) \to \Fin_*$, see \cite[\S 2.2.4]{HA} for details.
	In particular, we may identify the objects of $\Env(\cP)$ with finite collections $(c_i)_{i \in I}$ of colours of $\cP$; we will refer to $|I|$ as the \emph{length} of such a collection.
	The symmetric monoidal structure $\boxtimes$ on $\Env(\cP)$ is given by concatenating such collections.

	\textbf{Step 1: The adjunction is lax-idempotent.}
	Let $\cC^\otimes \to \Fin_*$ be a symmetric monoidal $(\infty,1)$-category.
	It follows as in \cite[Observation 2.1.2]{BarkanHaugsengea2025EnvelopesAlgebraicPatterns} that the counit $\varepsilon \colon \Env \iota \cC \to \cC$ is left adjoint to $\eta \colon \cC \to \Env \iota \cC$ as lax monoidal functors.
	By item \cref{laxi-adj4} of \cref{laxi-adj-conditions}, the adjunction $\Env \dashv \iota$ is lax-idempotent.

	\textbf{Step 2: $\otimes$-primes.}
	Let us write $\coAlg_{\Env}$ for $\coAlg_{\Env}(\CMon(\Cat_{(\infty,1)}))$ throughout this proof.
	To show that $\coAlg_{\Env}$ is equivalent to $\Op_\infty$, we first study what it means for a symmetric monoidal $(\infty,1)$-category $\cC$ that $\varepsilon \colon \Env \iota \cC \to \cC$ admits a further left adjoint $\nu$.
	In this case we have adjunctions
	\[\begin{tikzcd}[ampersand replacement=\&]
		{\Env(\iota \cC)} \&\& \cC
		\arrow[""{name=0, anchor=center, inner sep=0}, "\varepsilon"{description, pos=0.45}, from=1-1, to=1-3]
		\arrow[""{name=1, anchor=center, inner sep=0}, "\eta", curve={height=-18pt}, from=1-3, to=1-1]
		\arrow[""{name=2, anchor=center, inner sep=0}, "\nu"', curve={height=18pt}, from=1-3, to=1-1]
		\arrow["\dashv"{anchor=center, rotate=-90}, draw=none, from=0, to=1]
		\arrow["\dashv"{anchor=center, rotate=-90}, draw=none, from=2, to=0]
	\end{tikzcd}\]
	where $\nu$ and $\varepsilon$ are strong symmetric monoidal and $\eta$ is lax symmetric monoidal.
	Moreover, $\eta$ and hence $\nu$ are fully faithful (cf.\ \cref{lem:extra-adj-ff}).
	Note that the essential image of $\eta$ consists precisely of the length 1 objects.
	Let us call an object $c$ in $\cC$ an \emph{$\otimes$-prime}, or \emph{prime} for short, if whenever $c \simeq d \otimes e$, then exactly one of $d$ and $e$ is equivalent to the unit of $\cC$.
	We claim that an object $c$ of $\cC$ is prime if and only if $\nu(c) \simeq \eta(c)$.
	Namely, suppose $c$ is not prime.
	If $c \simeq \unit_\cC$, then $\nu(c)$ is the unit of $\Env(\cC)$ by strong monoidality, which has length $0 \neq 1$.
	Since $\nu$ is fully faithful and $\nu(\unit_\cC)$ is the unique object of length $0$ in $\Env(\cC)$, this also implies that $\nu(c)$ has length $\geq 1$ whenever $c$ is not the unit.
	In particular, if $c$ is of the form $d \otimes e$ where $d$ and $e$ are both not the unit, then $\nu(c) = \nu(d) \boxtimes \nu(e)$ must have length $>1$, hence $\nu(c) \neq \eta(c)$.
	If on the other hand $c$ is prime, then since $c \simeq \varepsilon \nu(c)$, it follows that $\nu(c)$ must be a collection of the form
	\[\nu(c) = (c,d_1,\cdots,d_n)\]
	such that $d_1 \otimes \cdots \otimes d_n = \unit_\cC$.
	It follows that $\nu(\unit_\cC) \simeq \nu(d_1) \boxtimes \cdots \boxtimes \nu(d_n)$ has length $0$, hence that $d_i \simeq \unit_\cC$ for every $i$.
	By the description of $\Env(\cC)$ in \cite[Remark 2.2.4.3]{HA} that if $n \geq 1$, then the map
	\[\varepsilon \colon \Map_{\Env(\cC)}(\nu(c), (c,\unit_\cC,\unit_\cC)) \to \Map_{\cC}(\varepsilon\nu(c), \varepsilon(c,\unit_\cC,\unit_\cC)) \simeq \Map_\cC(c,c)\]
	is not injective on path components. This contradicts that $\nu$ is left adjoint to $\varepsilon$, so we conclude that $n=0$.
	It follows that $\nu(c)$ has length 1 and hence lies in the essential image of $\eta$.
	This implies that the comparison map
	\[\nu(c) \xrightarrow{\mathrm{unit}} \eta \varepsilon \nu(c) \simeq \eta(c)\]
	is an equivalence, hence $c$ is $\otimes$-prime if and only if $\nu(c) \simeq \eta(c)$.
	We let $\cC^\mathrm{prime}$ denote the full sub-$\infty$-operad, in the sense of \cite[\S 2.2.1]{HA}, of $\iota \cC$ spanned by the $\otimes$-primes.

	\textbf{Step 3: Prime factorization.}
	Given a general $c$, one can view $\nu(c)$ as a ``prime factorization'' of $c$.
	To see this, note that $\nu(c)$ is of the form $\eta(d_1) \boxtimes \cdots \boxtimes \eta(d_n)$ for certain $d_1,\ldots,d_n$ in $\cC$, and hence that $c \simeq \varepsilon \nu(c) \simeq d_1 \otimes \cdots \otimes d_n$.
	By strong monoidality of $\varepsilon$ and $\nu$, it follows that
	\[\begin{tikzcd}[ampersand replacement=\&]
		\& {\nu\varepsilon(\eta(d_1) \boxtimes \cdots \boxtimes \eta(d_n))} \& {\eta(d_1) \boxtimes \cdots \boxtimes \eta(d_n)} \\
		{\nu(d_1) \boxtimes \cdots \boxtimes \nu(d_n)} \& {\nu \varepsilon\eta(d_1) \boxtimes \cdots \boxtimes \nu\varepsilon\eta(d_n)}
		\arrow["{\mathrm{counit}}", from=1-2, to=1-3]
		\arrow["\simeq", from=1-2, to=2-2]
		\arrow["{\mathrm{counit}^{\otimes n}}"', from=2-2, to=1-3]
		\arrow["\simeq"', from=2-2, to=2-1]
	\end{tikzcd}\]
	commutes.
	The top horizontal arrow is an equivalence since $\eta(d_1) \boxtimes \cdots \boxtimes \eta(d_n)$ lies in the image of $\nu$, hence the diagonal arrow is also an equivalence in $\Env(\cC)$.
	It follows that $\nu(d_i) \simeq \eta(d_i)$ for every $i$, hence that $d_i$ is a $\otimes$-prime.
	As a direct consequence, we see that the essential image of $\nu \colon \cC \to \Env(\cC)$ is precisely $\Env(\cC^\mathrm{prime})$.

	\textbf{Step 4: The equivalence $\Op_\infty \simeq \coAlg_{\Env}$.}
	By \cref{main-thm-2}, $\coAlg_{\Env}$ is the locally full sub-$(\infty,2)$-category of $\CMon(\Cat_{(\infty,1)})$ whose objects are those symmetric monoidal $(\infty,1)$-categories for which the extra left adjoint $\nu$ exists, and its 1-morphisms are the symmetric monoidal functors for which
	\[\begin{tikzcd}[ampersand replacement=\&]
		{\Env \cC} \& {\Env \cD} \\
		\cC \& \cD
		\arrow["{\Env f}", from=1-1, to=1-2]
		\arrow["\nu", from=2-1, to=1-1]
		\arrow["f", from=2-1, to=2-2]
		\arrow["\nu", from=2-2, to=1-2]
	\end{tikzcd}\]
	commutes.
	By the description of the prime objects in terms of $\nu$, it follows that any functor in $\coAlg_{\Env}$ preserves primes, hence $\cC \mapsto \cC^\mathrm{prime}$ defines a functor $\coAlg_{\Env} \to \Op_\infty$.
	Since the primes in $\Env(\cP)$ are the length 1 objects, it follows that $\Env(\cP)^\mathrm{prime}$ agrees with the image of $\cP^\otimes \hookrightarrow \iota \Env(\cP)$. In particular, there is a natural equivalence $\cP^\otimes \simeq \Env(\cP)^\mathrm{prime}$.
	On the other hand, the inclusion $\cC^\mathrm{prime} \hookrightarrow \iota \cC$ induces the strong monoidal functor $\varepsilon|_{\Env(\cC^\mathrm{prime})} \colon \Env(\cC^\mathrm{prime}) \to \cC$ which is an equivalence because it is inverse to $\nu$.
	We therefore see that $\Env \colon \Op_\infty \to \coAlg_{\Env}$ is an equivalence with inverse $(-)^\mathrm{prime}$.

	\textbf{Step 5: Monadicity.}
	Since the adjunction is lax-idempotent, to see that it is monadic it suffices, by \cref{main-thm-1}, to show the following:
	\begin{itemize}
		\item an $\infty$-operad $\cP^\otimes \to \Fin_*$ is a cocartesian fibration if $\eta \colon \cP^\otimes \to \Env(\cP)^\otimes$ admits a left adjoint $\varepsilon$ over $\Fin_*$, and
		\item a map of $\infty$-operads $f \colon \cP^\otimes \to \cQ^\otimes$ preserves cocartesian lifts of active morphisms if
		\[\begin{tikzcd}
			{\cP^\otimes} & {\cQ^\otimes} \\
			{\Env(\cP)^\otimes} & {\Env(\cQ)^\otimes}
			\arrow[from=1-1, to=1-2, "f"]
			\arrow[from=2-1, to=1-1,"\varepsilon"]
			\arrow[from=2-1, to=2-2, "{\Env(f)^\otimes}"]
			\arrow[from=2-2, to=1-2,"\varepsilon"]
		\end{tikzcd}\]
		commutes.
	\end{itemize}
	The first of these follows from \cite[Remark 2.2.4.12]{HA}, and the second follows from the characterization of the cocartesian morphisms given in item (2) of \cite[Lemma 2.2.4.11]{HA}.
\end{proof}

\section{Cocompletions}\label{sec:cocompletions}

Given a collection $I$ of $(\infty,1)$-categories, we will call a colimit an \emph{$I$-colimit} if it is indexed by an $(\infty,1)$-category that lies in $I$.
An $(\infty,1)$-category that admits all $I$-colimits will be called \emph{$I$-cocomplete}.
We will write $\Catco{I}$ for the $(\infty,2)$-category of small $I$-cocomplete $(\infty,1)$-categories and functors between them that preserve $I$-colimits.
Given two small collections $I \subset J$ of small $(\infty,1)$-categories, Lurie \cite[\S 5.3.6]{HTT} showed that the inclusion
\[\Catco{J} \hookrightarrow \Catco{I}\]
admits an explicit $(\infty,2)$-categorical left adjoint $\cP_I^J$.
We will show the following:

\begin{theorem}\label{thm:I-J-colimit-compl}
	Let $I \subset J$ be small collections of small $(\infty,1)$-categories.
	Then the adjunction
	\[\begin{tikzcd}[ampersand replacement=\&]
		{\cP_I^J : \Catco{I}} \&\& {\Catco{J} : \mathrm{incl}}
		\arrow[""{name=0, anchor=center, inner sep=0}, shift left=2, from=1-1, to=1-3]
		\arrow[""{name=1, anchor=center, inner sep=0}, hook', shift left=2, from=1-3, to=1-1]
		\arrow["\dashv"{anchor=center, rotate=-90}, draw=none, from=0, to=1]
	\end{tikzcd}\]
	is lax-idempotent and monadic.
	Moreover, the unit $j_\cC \colon \cC \to \cP_I^J(\cC)$ is fully faithful for any $\cC$ in $\Catco{I}$.
\end{theorem}

\begin{proof}
	It is shown in \cite[Proposition 5.3.6.2]{HTT} that $j_\cC \colon \cC \to \cP_I^J(\cC)$ is fully faithful for any $\cC$ in $\Catco{I}$.

	By \cref{main-thm-inverse-2}, the result will follow from the following two claims:

	\textbf{Claim 1:} \textit{An $I$-cocomplete $(\infty,1)$-category $\cC$ is $J$-cocomplete if and only if $j_\cC \colon \cC \to \cP_I^J(\cC)$ admits a left adjoint.}

	For the ``if'' direction, suppose that $j_\cC$ admits such a left adjoint $L_\cC$ and let $D \colon \cA \to \cC$ be a diagram with $\cA$ in $J$.
	Then the colimit of $D$ exists and is given by $L_\cC(\colim_{a \in \cA} j_\cC D(a))$.
	
	For the converse, suppose that $\cC$ is $J$-cocomplete. It follows from \cite[Proposition 5.3.6.2]{HTT} that there is a unique $J$-colimit preserving functor $L_\cC \colon \cP_I^J(\cC) \to \cC$ such that $L_\cC j_\cC \simeq \id_\cC$.
	Now let $\cE \subset \cP_I^J(\cC)$ be the full sub-$(\infty,1)$-category on those objects $X$ that have the following property:
	for any $c$ in $\cC$, the map
	\[L_\cC \colon \Map_{\cP_I^J(\cC)}(X,j_\cC(c)) \to \Map_{\cC}(L_\cC(X), L_\cC j_\cC (c)) \simeq \Map_\cC(L_\cC(X), c) \]
	is an equivalence.
	Then $L_\cC$ is left adjoint to $j_\cC$ if $\cE = \cP_I^J(\cC)$.
	Observe that $\cE \subset \cP_I^J(\cC)$ contains the essential image of $j_\cC$ and that it is closed under $J$-colimits since $L_\cC$ preserves $J$-colimits.
	It follows from the construction of $\cP_I^J(\cC)$ in the proof of \cite[Proposition 5.3.6.2]{HTT} that the smallest full sub-$(\infty,1)$-category with this property is $\cP_I^J(\cC)$ itself.

	\textbf{Claim 2:} \textit{Let $F \colon \cC \to \cD$ be an $I$-colimit preserving functor between $J$-cocomplete $(\infty,1)$-categories.
	Then $F$ preserves $J$-colimits if and only if the square
	\begin{equation}\label{eq:BC-rel-cocompl}\begin{tikzcd}[ampersand replacement=\&]
		{\cP_I^J(\cC)} \& {\cP_I^J(\cD)} \\
		\cC \& \cD
		\arrow["{\cP_I^J(F)}", from=1-1, to=1-2]
		\arrow["{L_\cC}"', from=1-1, to=2-1]
		\arrow[shorten <= 12, shorten >= 12,Rightarrow, from=1-2, to=2-1]
		\arrow["{L_\cD}", from=1-2, to=2-2]
		\arrow["F"', from=2-1, to=2-2]
	\end{tikzcd}\end{equation}
	commutes, where $L_\cC$ and $L_\cD$ denote the left adjoints of $j_\cC$ and $j_\cD$, respectively.}

	This follows by essentially the same argument as the first claim.
	If the square \cref{eq:BC-rel-cocompl} commutes, then $F$ preserves $J$-colimits by how these were constructed in the proof of Claim 1.
	On the other hand, if $F$ preserves $J$-colimits, then the collection of objects of $\cP_I^J(\cC)$ on which the Beck--Chevalley transformation \cref{eq:BC-rel-cocompl} is invertible is closed under $J$-colimits and contains $j_\cC(\cC)$, hence it is the whole $(\infty,1)$-category.
\end{proof}

\begin{example}[Finite colimit cocompletion]\label{exa:rex-monadic}
	$(\infty,1)$-categories that admit finite colimits and functors that preserve these are also known as \emph{right exact $(\infty,1)$-categories} and \emph{right exact functors}, respectively.
	Taking $I = \varnothing$ and letting $J = \mathrm{fin}$ consist of all finite $(\infty,1)$-categories in \cref{thm:I-J-colimit-compl}, it follows that $\Cat_{(\infty,1)}^\mathrm{rex}$ is monadic over $\Cat_{(\infty,1)}$, where the monad is given by the finite cocompletion functor $\cP^\mathrm{fin}$.
	Similarly, taking $J$ to consist of all $\kappa$-small $(\infty,1)$-categories for some infinite cardinal $\kappa$ yields the monadic lax-idempotent adjunction
	\[\begin{tikzcd}[ampersand replacement=\&]
		{\cP^\kappa : \Cat_{(\infty,1)}} \&\& {\Catco{\kappa} : \mathrm{incl}.}
		\arrow[""{name=0, anchor=center, inner sep=0}, shift left=2, from=1-1, to=1-3]
		\arrow[""{name=1, anchor=center, inner sep=0}, hook', shift left=2, from=1-3, to=1-1]
		\arrow["\dashv"{anchor=center, rotate=-90}, draw=none, from=0, to=1]
	\end{tikzcd}\]
\end{example}

\begin{example}[Coproduct cocompletion]\label{exa:coprod-cocompletion}
	If in the previous example one takes $J$ to consist of all finite discrete $(\infty,1)$-categories, then $\cP^\sqcup \coloneqq \cP^J$ is the finite coproduct completion functor and one obtains a monadic lax-idempotent adjunction
	\[\begin{tikzcd}[ampersand replacement=\&]
		{\cP^\sqcup : \Cat_{(\infty,1)}} \&\& {\Catco{\sqcup} : \mathrm{incl}.}
		\arrow[""{name=0, anchor=center, inner sep=0}, shift left=2, from=1-1, to=1-3]
		\arrow[""{name=1, anchor=center, inner sep=0}, hook', shift left=2, from=1-3, to=1-1]
		\arrow["\dashv"{anchor=center, rotate=-90}, draw=none, from=0, to=1]
	\end{tikzcd}\]
\end{example}

\begin{example}[Cocompletion under small colimits]
	To discuss cocompletion under all small colimits, let $\cU \lneq \cV$ be inaccessible cardinals. Call an $(\infty,1)$-category \emph{small} if it is $\cU$-small, \emph{large} if it is $\cV$-small and otherwise \emph{huge}.
	Let us write $\widehat{\Cat}_{(\infty,1)}$ for the (huge) $(\infty,2)$-category of large $(\infty,1)$-categories, and $\widehat{\Cat}{}_{(\infty,1)}^{\mathrm{cocomp}}$ for the locally full sub-$(\infty,2)$-category of $(\infty,1)$-categories that admit small colimits and functors between them that preserve these.
	Applying \cref{exa:rex-monadic} to the cardinal $\cU$, it follows that there is a monadic lax-idempotent adjunction
	\[\begin{tikzcd}[ampersand replacement=\&]
		{\cP : \widehat{\Cat}_{(\infty,1)}} \&\& {\widehat{\Cat}{}_{(\infty,1)}^{\mathrm{cocomp}} : \mathrm{incl}.}
		\arrow[""{name=0, anchor=center, inner sep=0}, shift left=2, from=1-1, to=1-3]
		\arrow[""{name=1, anchor=center, inner sep=0}, hook', shift left=2, from=1-3, to=1-1]
		\arrow["\dashv"{anchor=center, rotate=-90}, draw=none, from=0, to=1]
	\end{tikzcd}\]
	The left adjoint $\cP(\cC)$ agrees with the usual presheaf $(\infty,1)$-category $\Fun(\cC^\op, \Spc)$ whenever $\cC$ is small, but in general it is the full sub-$(\infty,1)$-category of $\Fun(\cC^\op, \widehat{\Spc})$ spanned by the small colimits of representables.
	Here $\widehat{\Spc}$ denotes the (huge) $(\infty,1)$-category of large spaces.
\end{example}

\begin{example}[Ind-cocompletion]\label{exa:ind-monadic}
	Letting $J$ be the collection of all small $\kappa$-filtered $(\infty,1)$-categories, where $\kappa$ is a (small) regular cardinal, we obtain the monadic lax-idempotent adjunction
	\begin{equation}\label{eq:Ind-adjunction}\begin{tikzcd}[ampersand replacement=\&]
		{\Ind_\kappa : \widehat{\Cat}_{(\infty,1)}} \&\& {\widehat{\Cat}{}_{(\infty,1)}^{\kappa\mhyphen\mathrm{fil}} : \mathrm{incl}.}
		\arrow[""{name=0, anchor=center, inner sep=0}, shift left=2, from=1-1, to=1-3]
		\arrow[""{name=1, anchor=center, inner sep=0}, hook', shift left=2, from=1-3, to=1-1]
		\arrow["\dashv"{anchor=center, rotate=-90}, draw=none, from=0, to=1]
	\end{tikzcd}\end{equation}
	Taking $I$ to consist of all $\kappa$-small $(\infty,1)$-categories and $J$ to consist of all small $(\infty,1)$-categories, it follows from \cite[Example 5.3.6.8]{HTT} that $\cP_I^J = \Ind_\kappa$ and that the adjunction \cref{eq:Ind-adjunction} restricts to a monadic lax-idempotent adjunction
	\[\begin{tikzcd}[ampersand replacement=\&]
		{\Ind_\kappa : \widehat{\Cat}{}_{(\infty,1)}^{\kappa\mhyphen\mathrm{cocomp}}} \&\& {\widehat{\Cat}{}_{(\infty,1)}^{\mathrm{cocomp}} : \mathrm{incl}.}
		\arrow[""{name=0, anchor=center, inner sep=0}, shift left=2, from=1-1, to=1-3]
		\arrow[""{name=1, anchor=center, inner sep=0}, hook', shift left=2, from=1-3, to=1-1]
		\arrow["\dashv"{anchor=center, rotate=-90}, draw=none, from=0, to=1]
	\end{tikzcd}\]
\end{example}

By dualizing the above, we also obtain results about free limit completions: Given a collection $I$ of $(\infty,1)$-categories, write $\Cat_{(\infty,1)}^{I\mhyphen\mathrm{comp}}$ for the $(\infty,2)$-category of small $I$-complete $(\infty,1)$-categories and functors between them that preserve $I$-limits.
We will write $\bar{I}$ for the collection of $(\infty,1)$-categories $\cC$ such that $\cC^\op \in I$.
Then it follows formally that the inclusion
\[\Cat_{(\infty,1)}^{J\mhyphen\mathrm{comp}} \hookrightarrow \Cat_{(\infty,1)}^{I\mhyphen\mathrm{comp}}\]
is right adjoint to the functor $\cC \mapsto \cP_{\bar{I}}^{\bar{J}}(\cC^\op)^\op$.
Moreover, the following is a formal consequence of \cref{thm:I-J-colimit-compl}:

\begin{corollary}\label{cor:limit-completion}
	Let $I \subset J$ be small collections of small $(\infty,1)$-categories.
	Then the adjunction
	\[\begin{tikzcd}[ampersand replacement=\&]
		{\cP_{\bar{I}}^{\bar{J}}((-)^\op)^\op : \Cat_{(\infty,1)}^{I\mhyphen\mathrm{comp}}} \&\& {\Cat_{(\infty,1)}^{J\mhyphen\mathrm{comp}} : \mathrm{incl}}
		\arrow[""{name=0, anchor=center, inner sep=0}, shift left=2, from=1-1, to=1-3]
		\arrow[""{name=1, anchor=center, inner sep=0}, hook', shift left=2, from=1-3, to=1-1]
		\arrow["\dashv"{anchor=center, rotate=-90}, draw=none, from=0, to=1]
	\end{tikzcd}\]
	is oplax-idempotent and monadic.
	Moreover, the unit $j_\cC \colon \cC \to \cP_{\bar{I}}^{\bar{J}}(\cC^\op)^\op$ is fully faithful for any $\cC$ in $\Cat_{(\infty,1)}^{I\mhyphen\mathrm{comp}}$. \qed
\end{corollary}

In the construction of the symmetric topos in \cref{sec:symtop}, the following example is important.

\begin{example}[Finite limit completion]\label{exa:lex-cats}
	$(\infty,1)$-categories that admit finite limits and functors that preserve these are usually called \emph{left exact}, and we will denote their $(\infty,2)$-category by $\Cat_{(\infty,1)}^\lex$.
	Taking $I = \varnothing$ and $J = \mathrm{fin}$ to consist of all finite $(\infty,1)$-categories in \cref{cor:limit-completion}, we obtain an oplax-idempotent monadic adjunction
	\[\begin{tikzcd}[ampersand replacement=\&]
		{(-)^\lex : \Cat_{(\infty,1)}} \&\& {\Cat_{(\infty,1)}^\lex : \mathrm{incl}.}
		\arrow[""{name=0, anchor=center, inner sep=0}, shift left=2, from=1-1, to=1-3]
		\arrow[""{name=1, anchor=center, inner sep=0}, hook', shift left=2, from=1-3, to=1-1]
		\arrow["\dashv"{anchor=center, rotate=-90}, draw=none, from=0, to=1]
	\end{tikzcd}\]
	In particular, it follows that $\cC$ admits finite limits if and only if $\cC \hookrightarrow \cC^\lex$ admits a right adjoint $\rho_\cC$, and that $f \colon \cC \to \cD$ preserves finite limits if and only if the square
	\[\begin{tikzcd}
		{\cC^\lex} & {\cD^\lex} \\
		\cC & \cD
		\arrow["{f^\lex}", from=1-1, to=1-2]
		\arrow["{\rho_\cC}"', from=1-1, to=2-1]
		\arrow["{\rho_\cD}", from=1-2, to=2-2]
		\arrow["f"', from=2-1, to=2-2]
	\end{tikzcd}\]
	commutes.
\end{example}

\begin{remark}
	We expect there to be analogues of \cref{thm:I-J-colimit-compl,cor:limit-completion} for enriched $(\infty,1)$-categories and weighted (co)limits in the sense of \cite{Heine2024HigherAlgebraWeighted}, and internal $(\infty,1)$-categories and internal colimits in the sense of \cite{MartiniWolf2024ColimitsCocompletionsInternal}, with essentially the same proofs.
\end{remark}

\section{Compactly assembled \texorpdfstring{$(\infty,1)$}{(∞,1)}-categories and continuous algebras}\label{sec:CompAss}

In \cref{exa:ind-monadic} we saw that a (potentially large) $(\infty,1)$-category $\cC$ admits filtered colimits precisely if the Yoneda embedding $\cC \hookrightarrow \Ind(\cC)$ admits a left adjoint.
If this left adjoint admits a further left adjoint $\hat y$, then, following \cite{JohnstoneJoyal1982ContinuousCategoriesExponentiable}, we call $\cC$ a \emph{continuous $(\infty,1)$-category}.\footnote{In the higher-categorical literature, the term \emph{compactly assembled $(\infty,1)$-category} is more common than continuous $(\infty,1)$-category. However, the definition of compact assembly typically includes that $\cC$ is accessible \cite{SAG,efimovKtheoryLocalizingInvariants2025,Ramzi2024DualizablePresentableinftycategories}, or even that $\cC$ is presentable \cite{KNPSheavesOnManfiolds}. Since it will be convenient to not include this for now, we opted to use the term \emph{continuous} $(\infty,1)$-category. However, we will discuss the accessible/presentable case at the end of this section.}
We will write $\widehat{\Cat}{}_{(\infty,1)}^{\mathrm{cont}}$ for the $(\infty,2)$-category of continuous $(\infty,1)$-categories and filtered colimit preserving functors $F \colon \cC \to \cD$ between them for which the mate square
\[\begin{tikzcd}[ampersand replacement=\&]
	\cC \& \cD \\
	{\Ind(\cC)} \& {\Ind(\cD)}
	\arrow["F", from=1-1, to=1-2]
	\arrow["{\hat y}"', from=1-1, to=2-1]
	\arrow[Rightarrow, shorten >= 6, shorten <= 12, from=1-2, to=2-1]
	\arrow["{\hat y}", from=1-2, to=2-2]
	\arrow["{\Ind(F)}"', from=2-1, to=2-2]
\end{tikzcd}\]
commutes strongly; we call such functors \emph{strongly continuous}.
This is an example of the general notion of a \emph{continuous algebra} over a lax-idempotent monad.

\begin{definition}[Continuous $T$-algebras]\label{contalg}
	Let $T$ be a lax-idempotent monad on an $(\infty,2)$-category $\sX$.
	Then a \emph{continuous $T$-algebra} is an object $x$ in $\sX$ such that the unit $\eta_x \colon x \to Tx$ admits a reflective left adjoint $\alpha_x \colon Tx \to x$ which admits a further left adjoint $\nu_x \colon x \to Tx$.
	Given continuous $T$-algebras $x$ and $y$ in $\sX$, a map $f \colon x \to y$ is called a \emph{morphism of continuous $T$-algebras} if the mate squares
	\[\begin{tikzcd}[ampersand replacement=\&]
		Tx \& Ty \\
		x \& y
		\arrow["Tf", from=1-1, to=1-2]
		\arrow["{\alpha_x}"', from=1-1, to=2-1]
		\arrow[shorten >= 9, shorten <= 6, Rightarrow, from=1-2, to=2-1]
		\arrow["{\alpha_y}", from=1-2, to=2-2]
		\arrow["f"', from=2-1, to=2-2]
	\end{tikzcd}
	\quad \text{and} \quad
	\begin{tikzcd}[ampersand replacement=\&]
		x \& y \\
		Tx \& Ty
		\arrow["f", from=1-1, to=1-2]
		\arrow["{\nu_x}"', from=1-1, to=2-1]
		\arrow[shorten >= 6, shorten <= 9, Rightarrow, from=1-2, to=2-1]
		\arrow["{\nu_y}", from=1-2, to=2-2]
		\arrow["Tf"', from=2-1, to=2-2]
	\end{tikzcd}\]
	both commute.
	Note that the commutativity of the first square is used to construct the mate of the second square.
	We will write $\sX^{TT}$ for the locally full sub-$(\infty,2)$-category of $\sX$ spanned by the continuous $T$-algebras and morphisms of continuous $T$-algebras.
\end{definition}

It is clear that $\sX^{TT}$ is a locally full sub-$(\infty,2)$-category of the $(\infty,2)$-category $\sX^T \simeq \Alg_T(\sX)$ (see \cref{talg,main-thm-1}).
Moreover, since
\[\begin{tikzcd}[ampersand replacement=\&]
	{T : \sX} \&\& {\Alg_T(\sX) : U}
	\arrow[""{name=0, anchor=center, inner sep=0}, shift left=2, from=1-1, to=1-3]
	\arrow[""{name=1, anchor=center, inner sep=0}, hook', shift left=2, from=1-3, to=1-1]
	\arrow["\dashv"{anchor=center, rotate=-90}, draw=none, from=0, to=1]
\end{tikzcd}\]
is a lax-idempotent adjunction, it follows from \cref{laxi-adj-conditions} that $TU$ is a lax-idempotent comonad.

\begin{proposition}\label{contalg-thm}
	Let $T$ be a lax-idempotent monad on an $(\infty,2)$-category $\sX$.
	The inclusion $\sX^{TT} \hookrightarrow \Alg_T(\sX)$ of the $(\infty,2)$-category of continuous $T$-algebras into the $(\infty,2)$-category of $T$-algebras is comonadic, with associated (lax-idempotent) comonad $TU$.
\end{proposition}

\begin{proof}
	By \cref{lem:extra-adj-ff} and \cref{main-thm-2}, we may identify $\coAlg_{TU}(\Alg_T(\sX))$ with the full sub-$(\infty,2)$-category of $\sX^{TT}$ spanned by those continuous $T$-algebras $x$ for which $\nu_x \colon x \to Tx$ lies in $\Alg_T(\sX)$.
	By \cref{adjfunthm}, it follows that a left adjoint 1-morphism in $\sX$ between $T$-algebras always lies in $\Alg_T(\sX)$, hence this full sub-$(\infty,2)$-category is $\sX^{TT}$ itself.
\end{proof}

\begin{example}[$\infty$-operads as continuous algebras]
	Let $\CMon^\lax(\Cat_{(\infty,1)})$ denote the $(\infty,2)$-category of symmetric monoidal $(\infty,1)$-categories and lax monoidal functors between them.
	It follows from the results in \cref{sec:operads} that $\infty$-operads are the continuous algebras for the lax-idempotent monad $\Env$ on $\CMon^\lax(\Cat_{(\infty,1)})$.
\end{example}

\begin{example}[Free algebras are continuous]\label{exa:Free-talg-is-continuous}
	Let $T$ be a lax-idempotent monad on an $(\infty,2)$-category $\sX$.
	Then any free $T$-algebra is a continuous $T$-algebra; more generally, the free functor $T \colon \sX \to \Alg_T(\sX)$ lands in the locally full sub-$(\infty,2)$-category $\sX^{TT}$ of continuous $T$-algebras.
	This can be seen directly from the definition of the $(\infty,2)$-category $\sX^{TT}$ and the fact that by \cref{alt}, $T(\eta_x)$ is left adjoint to $\mu_x$.
	Alternatively, this follows by combining \cref{contalg-thm} and the dual of \cite[Proposition 4.7.3.3.(1)]{HA}: the latter implies that $T$ lifts along $\coAlg_{TU}(\Alg_T(\sX)) \to \Alg_T(\sX)$ (at least at the level of underlying $(\infty,1)$-categories), which by \cref{contalg-thm} implies that the free functor $\sX \to \Alg_T(\sX)$ must land in the locally full sub-$(\infty,2)$-category $\sX^{TT}$.
\end{example}

\begin{corollary}
	The $(\infty,2)$-category of continuous $(\infty,1)$-categories is the $(\infty,2)$-category of continuous algebras for the lax-idempotent monad $\Ind$ on the $(\infty,2)$-category $\widehat{\Cat}_{(\infty,1)}$ of large $(\infty,1)$-categories.
	In particular, the adjunction
	\begin{equation}\begin{tikzcd}[ampersand replacement=\&]
		{\mathrm{incl} : \widehat{\Cat}{}_{(\infty,1)}^{\mathrm{cont}}} \&\& {\widehat{\Cat}{}_{(\infty,1)}^{\mathrm{fil}} : \Ind}
		\arrow[""{name=1, anchor=center, inner sep=0}, shift left=2, from=1-3, to=1-1]
		\arrow[""{name=0, anchor=center, inner sep=0}, hook', shift left=2, from=1-1, to=1-3]
		\arrow["\dashv"{anchor=center, rotate=-90}, draw=none, from=0, to=1]
	\end{tikzcd}\end{equation}
	is comonadic and lax-idempotent.
\end{corollary}

\begin{proof}
	Since continuous $(\infty,1)$-categories are by definition the continuous algebras for the lax-idempotent monad $\Ind(-)$ on $\widehat{\Cat}_{(\infty,1)}$, it follows from \cref{contalg-thm} that the inclusion $\widehat{\Cat}{}_{(\infty,1)}^{\mathrm{cont}} \hookrightarrow \widehat{\Cat}{}_{(\infty,1)}^{\mathrm{fil}}$ is comonadic and that the associated comonad is lax-idempotent.
	The adjunction is lax-idempotent by \cref{laxi-adj-conditions}.
\end{proof}

Following \cite{SAG,efimovKtheoryLocalizingInvariants2025}, we call an $(\infty,1)$-category \emph{compactly assembled} if it is an accessible and continuous $(\infty,1)$-category.
We write $\mathrm{CompAss}$ for the full sub-$(\infty,2)$-category of $\widehat{\Cat}{}_{(\infty,1)}^{\mathrm{cont}}$ spanned by the compactly assembled $(\infty,1)$-categories.
Because of size reasons, compactly assembled $(\infty,1)$-categories are not the continuous algebras for a lax-idempotent monad on $\widehat{\Cat}_{(\infty,1)}$.
However, one still has the following comonadicity result, where $\mathrm{Acc}^\mathrm{fil}_{\aleph_1}$ denotes the $(\infty,2)$-category of $\aleph_1$-accessible $(\infty,1)$-categories that admit filtered colimits, and functors between them that preserve filtered colimits and $\aleph_1$-compact objects.

\begin{proposition}\label{prop:CompAss-comonadic}
	The inclusion
	\[\mathrm{CompAss} \hookrightarrow \mathrm{Acc}^\mathrm{fil}_{\aleph_1}\]
	is comonadic, with associated comonad $\Ind((-)^{\aleph_1})$. 
	Moreover, this comonad is lax-idempotent.
\end{proposition}

\begin{proof}
	It is shown in \cite[Proposition 1.24]{efimovKtheoryLocalizingInvariants2025} that any compactly assembled $(\infty,1)$-category is $\aleph_1$-accessible.
	Moreover, any strongly continuous functor between compactly assembled $(\infty,1)$-categories preserves $\aleph_1$-compact objects by the description of the $\aleph_1$-compact objects given in loc.\ cit.\ and the definition of strong continuity \cite[Definition 1.10]{efimovKtheoryLocalizingInvariants2025}.
	It follows that the inclusion $\mathrm{CompAss} \to \widehat{\Cat}{}_{(\infty,1)}$ lands in $\mathrm{Acc}^\mathrm{fil}_{\aleph_1}$.

	Observe that if $\cC$ is $\aleph_1$-accessible and admits filtered colimits, then the colimit functor $c \colon \Ind(\cC) \to \cC$ factors as
	\[\Ind(\cC) \xrightarrow{R} \Ind(\cC^{\aleph_1}) \xrightarrow{c'} \cC\]
	where the first arrow is right adjoint to the inclusion $\iota \colon \Ind(\cC^{\aleph_1}) \hookrightarrow \Ind(\cC)$ and the second arrow is the colimit functor.
	In particular, applying the dual of \cref{lem:lift-along-lali}, it follows that $c' = c \circ \iota$ admits a left adjoint $\hat{y}'$ if and only if $c$ admits a left adjoint $\hat{y}$, and that in this case $\hat{y} = \iota \circ \hat{y}'$.

	From this, it follows by the same argument as \cite[Proposition 1.89]{efimovKtheoryLocalizingInvariants2025} that $\Ind((-)^{\aleph_1})$ is right adjoint to the inclusion $\mathrm{CompAss} \hookrightarrow \mathrm{Acc}^\mathrm{fil}_{\aleph_1}$. Moreover, combining this statement about $\hat{y}'$ with (the dual of) \cref{main-thm-inverse-2}, it follows that this inclusion is comonadic and the corresponding comonad on $\mathrm{Acc}^\mathrm{fil}_{\aleph_1}$ is lax-idempotent.
\end{proof}

If $\cC$ is a stable compactly assembled $(\infty,1)$-category, then since it admits finite colimits, it must admit all colimits by \cite[Corollary 4.2.3.11]{HTT} and is therefore presentable.
By \cite[Proposition D.7.3.1]{SAG}, these are precisely the dualizable objects in $\preslst$ for Lurie's tensor product \cite[\S 4.8.2]{HA}, hence they are typically called dualizable (stable) $(\infty,1)$-categories.
It was shown in \cite[Theorem B]{Ramzi2024DualizablePresentableinftycategories} that $(\preslst)^\mathrm{dbl}$ is comonadic over $\presl_{\mathrm{St},\aleph_1}$, where $(\preslst)^\mathrm{dbl}$ is the locally full sub-$(\infty,2)$-category of $\preslst$ spanned by dualizable objects and internal left adjoints, and $\presl_{\mathrm{St},\aleph_1}$ is the locally full sub-$(\infty,2)$-category of $\preslst$ spanned by $\aleph_1$-compactly generated stable $(\infty,1)$-categories and left adjoint functors that preserve $\aleph_1$-compact objects.
We obtain an alternative proof of this result by restricting the comonadic adjunction from \cref{prop:CompAss-comonadic}:

\begin{proposition}[cf.\ {\cite[Theorem B]{Ramzi2024DualizablePresentableinftycategories}}]
	The adjunction
	\[\begin{tikzcd}
	{\mathrm{incl} : (\preslst)^\mathrm{dbl}} && {\presl_{\mathrm{St},\aleph_1} : \Ind((-)^{\aleph_1})}
	\arrow[""{name=1, anchor=center, inner sep=0}, shift left=2, from=1-3, to=1-1]
	\arrow[""{name=0, anchor=center, inner sep=0}, hook', shift left=2, from=1-1, to=1-3]
	\arrow["\dashv"{anchor=center, rotate=-90}, draw=none, from=0, to=1]
	\end{tikzcd}\]
	is comonadic and lax-idempotent.
\end{proposition}

\begin{proof}
	This follows from \cref{prop:CompAss-comonadic} since $\Ind{((-)^{\aleph_1})}$ preserves stable $(\infty,1)$-categories and colimit preserving functors between them.
\end{proof}

\begin{remark}
	More generally, let $\cV$ be a presentably symmetric monoidal $(\infty,1)$-category.
	Recall that the $(\infty,2)$-category $\presl_\cV$ of presentable $\cV$-enriched $(\infty,1)$-categories is equivalent to the $(\infty,2)$-category $\Mod_\cV(\presl)$ of $\cV$-modules in $\presl$.
	It is shown in \cite[Theorem 3.4]{Ramzi2024DualizablePresentableinftycategories} that there exists a cardinal $\kappa$ (depending on $\cV$) such that the inclusion $(\presl_\cV)^\mathrm{dbl} \hookrightarrow \presl_\cV$ factors through $\Mod_\cV(\presl_\kappa)$, where $\presl_\kappa$ denotes the $(\infty,2)$-category of $\kappa$-compactly generated $(\infty,1)$-categories.
	By the characterization of dualizable objects in $\presl_\cV$ given in item 4 of \cite[Theorem 1.49]{Ramzi2024DualizablePresentableinftycategories} and \cref{main-thm-inverse-2} it then follows that the inclusion $(\presl_\cV)^\mathrm{dbl} \hookrightarrow \Mod_\cV(\presl_\kappa)$ is comonadic.
	In particular, this gives an alternative proof of Ramzi's Theorem 3.7 \cite{Ramzi2024DualizablePresentableinftycategories}, though it still uses \cite[Theorem 3.4]{Ramzi2024DualizablePresentableinftycategories}, which asserts that the inclusion $(\presl_\cV)^\mathrm{dbl} \hookrightarrow \presl_\cV$ lands in $\Mod_\cV(\presl_\kappa)$.
\end{remark}

\section{The symmetric \texorpdfstring{$(\infty,1)$}{(∞,1)}-topos}\label{sec:symtop}

Let us write $\LTop$ for the $(\infty,2)$-category of $(\infty,1)$-toposes and left exact left adjoints between them.
In this section we construct a left adjoint to the forgetful functor $\LTop \to \presl$, which sends a presentable $(\infty,1)$-category to its \emph{symmetric $(\infty,1)$-topos}.
This left adjoint was also recently constructed by Lehner \cite[Theorem 18.33]{Lehner2026ShapeTheory}.
We then show that the resulting adjunction is oplax-idempotent and monadic.
This can be viewed as a higher-categorical analogue of \cite[\S 4.2]{2006SingularCoveringsToposes}.

\begin{definition}[The symmetric $(\infty,1)$-topos]\label{def:sym-topos}
	Let $\cC \in \Cat_{(\infty,1)}^{\kappa\text{-}\cocomp}$ for $\kappa$ a regular cardinal, and let $\Sigma^{\kappa}_{\cC}$ denote the set of morphisms in $\cP(\cC^{\lex})$:
	\[
		\Sigma^{\kappa}_{\cC} = \bigl\{\colim_{i \in I} y_{c_i} \to y_c \;\big|\; c = \colim_{i \in I} c_i \text{ in } \cC,\; |I| < \kappa \bigr\}.
	\]
	The \emph{symmetric $(\infty,1)$-topos} $\Sym^{\kappa}(\cC)$ is the localization of $\cP(\cC^{\lex})$ at the strongly saturated, base-change stable class generated by $\Sigma^{\kappa}_{\cC}$. By \cite[Propositions 6.2.1.1 \& 6.2.1.2]{HTT}, the localization
	\[
		L_{\Sigma_{\cC}^{\kappa}} : \cP(\cC^{\lex}) \llra \Sym^{\kappa}(\cC) : R_{\Sigma_{\cC}^{\kappa}}
	\]
	is left exact, and in particular $\Sym^{\kappa}(\cC)$ is an $(\infty,1)$-topos.
\end{definition}

\begin{remark}\label{rem:remsym}
	Note that by construction we can identify the mapping $(\infty,1)$-category $\twomap_{\LTop}(\Sym^{\kappa}(\cC),\cX)$ with the full sub-$(\infty,1)$-category of $\twomap_{\LTop}(\cP(\cC^{\lex}),\cX)$ on those functors which invert the morphisms in $\Sigma^{\kappa}_{\cC}$. In particular, the restriction map
	\[
		\twomap_{\LTop}(\Sym^{\kappa}(\cC),\cX) \to \Fun(\cC,\cX)
	\]
	along $\cC \to \cP(\cC^{\lex})$ factors through the full sub-$(\infty,1)$-category on $\kappa$-small colimit preserving functors.
\end{remark}

\begin{construction}\label{cons:etasmall}
	Let $\cC \in \Cat_{(\infty,1)}^{\kappa\text{-}\cocomp}$ for $\kappa \lneq \mathcal{U}$. We consider the composite functor
	\[
		\eta_{\cC} \colon \cC \to \cP(\cC) \to \cP(\cC^{\lex}) \xrightarrow{L_{\Sigma_{\cC}^{\kappa}}} \Sym^{\kappa}(\cC)
	\]
	which preserves the $\kappa$-small colimits in $\cC$ by construction.
\end{construction}

\begin{proposition}\label{prop:symtoposadj}
	The forgetful functor $U \colon \LTop \to \presl$ admits a left adjoint, sending a presentable $(\infty,1)$-category $\cC=\Ind_{\kappa}(\cC_\kappa)$ with $\cC_\kappa \in \Cat_{(\infty,1)}^{\kappa\text{-}\cocomp}$ to $\Sym^{\kappa}(\cC_\kappa)$.
\end{proposition}
\begin{proof}
	We show that $\twomap_{\presl}(\cC,U(-))$ is represented by $\Sym^{\kappa}(\cC_{\kappa})$. By \cite[Proposition 5.3.5.10]{HTT}, the map $\eta_{\cC_{\kappa}}$ from \cref{cons:etasmall} extends uniquely to a colimit preserving functor $\eta_{\cC} \colon \cC \to \Sym^{\kappa}(\cC_{\kappa})$. We claim that restriction along $\eta_{\cC}$ induces an equivalence of $(\infty,1)$-categories
	\[
		\eta_{\cC}^* \colon \twomap_{\LTop}(\Sym^{\kappa}(\cC_{\kappa}),\cX) \xrightarrow{\simeq} \twomap_{\presl}(\cC,\cX),
	\]
	for every $(\infty,1)$-topos $\cX$. By the universal property of $\Ind_{\kappa}$, this is equivalent to showing that restriction along $\eta_{\cC_\kappa}$ is fully faithful with essential image $\Fun^{\kappa}(\cC_{\kappa},\cX)$, i.e. the full sub-$(\infty,1)$-category on those functors which preserve $\kappa$-small colimits. To this end, consider the composite
	\[
		\twomap_{\LTop}(\Sym^{\kappa}(\cC_{\kappa}),\cX) \xrightarrow{L^{*}_{\Sigma_{\cC_\kappa}^{\kappa}}} \twomap_{\LTop}(\cP(\cC_\kappa^{\lex}),\cX)\xrightarrow{y^*_{\cC_{\kappa}^{\lex}}}\Fun(\cC_{\kappa}^{\lex},\cX),
	\]
	where the second map is restriction along the Yoneda embedding. By \cite[Proposition 6.1.5.2]{HTT}, $y^*_{\cC_{\kappa}^{\lex}}$ is fully faithful with essential image the finite-limit-preserving functors, yielding an equivalence $\twomap_{\LTop}(\cP(\cC_\kappa^{\lex}),\cX) \simeq \Fun(\cC_\kappa,\cX)$. By \cref{rem:remsym}, we obtain a commutative diagram
	\[
		\begin{tikzcd}
			\twomap_{\LTop}(\Sym^{\kappa}(\cC_{\kappa}),\cX) \arrow[d] \arrow[r] & \Fun^{\kappa}(\cC_{\kappa},\cX) \arrow[d] \\
			\twomap_{\LTop}(\cP(\cC_\kappa^{\lex}),\cX) \arrow[r,"\simeq"'] & \Fun(\cC_\kappa,\cX),
		\end{tikzcd}
	\]
	which is easily verified to be a pullback, so the bottom equivalence restricts to an equivalence on the top row, as required.
\end{proof}

\begin{notation}
	We denote by $\Sym$ the left adjoint to the forgetful functor $U \colon \LTop \to \presl$ from \cref{prop:symtoposadj}. Note that due to the previous proposition $\Sym$ is independent of the choice of presentation $\Ind_\kappa(\cC_\kappa) \simeq \cC$.
\end{notation}

\begin{proposition}\label{prop:topos-unit-char}
	Let $\cC \in \presl$. Then $\cC$ is an $(\infty,1)$-topos if and only if the unit $\eta_{\cC} \colon \cC \to \Sym(\cC)$ is fully faithful and admits a right adjoint in $\presl$.
\end{proposition}
\begin{proof}
	For the ``if'' direction, suppose that $\eta_{\cC}$ is fully faithful with right adjoint $\omega_{\cC} \colon \Sym(\cC) \to \cC$ in $\presl$. Then $\omega_{\cC}$ is in particular a left exact left adjoint. We therefore see by \cref{lem:extra-adj-ff} that $\cC$ is a left exact reflective localization of an $(\infty,1)$-topos, hence it is an $(\infty,1)$-topos. 

	Conversely, suppose that $\cC$ is an $(\infty,1)$-topos. As shown in the proof of \cite[Proposition 6.1.5.3]{HTT}, there exists a regular cardinal $\kappa$ such that the full sub-$(\infty,1)$-category $i \colon \cC_\kappa \subset \cC$ spanned by the $\kappa$-compact objects admits finite limits, $\cC = \Ind_\kappa(\cC_\kappa)$ and such that $\cC$ is a left exact localization of $\cP(\cC_\kappa)$.
	Let $j \colon \cC_\kappa \to \cC_{\kappa}^{\lex}$ denote the canonical map, and observe that by \cref{exa:lex-cats}, it admits a right adjoint $j \dashv \rho$ such that $\rho \circ j \simeq \id$.
	The composite
     \[
		\cP(\cC_\kappa^{\lex}) \xrightarrow{\rho_!} \cP(\cC_\kappa) \xrightarrow{} \cC
	 \]
	 is a left exact left adjoint which inverts $\Sigma_{\cC_{\kappa}}^{\kappa}$ and thus descends to a left exact left adjoint $\omega_\cC \colon \Sym^{\kappa}(\cC_\kappa) \to \cC$. We claim that the following diagram commutes:
	 \[\begin{tikzcd}
	{\cP(\cC_\kappa^{\lex})} & {\cP(\cC_\kappa)} & {\cP(\cC_\kappa^{\lex})} \\
	{\Sym^{\kappa}(\cC_\kappa)} & \cC & {\Sym^{\kappa}(\cC_{\kappa})}{.}
	\arrow["{\rho_!}", from=1-1, to=1-2]
	\arrow[from=1-1, to=2-1]
	\arrow["{j_!}", from=1-2, to=1-3]
	\arrow[from=1-2, to=2-2]
	\arrow[from=1-3, to=2-3]
	\arrow["{\omega_\cC}"', from=2-1, to=2-2]
	\arrow["{\eta_{\cC}}"', from=2-2, to=2-3]
\end{tikzcd}\]
The left-hand square commutes by construction. For the right-hand square, note that every functor involved preserves colimits, so it suffices to check commutativity after restricting to $\cC_\kappa$, where it is immediate.

Commutativity of the diagram immediately gives $\omega_\cC \circ \eta_\cC \simeq \id$. Moreover, the counit of $j_! \dashv \rho_!$ produces a natural transformation $\epsilon \colon \eta_\cC \circ \omega_\cC \to \id$. The result now follows from \cref{easyadjun} after unwinding the definitions.
\end{proof}

\begin{lemma}\label{lem:BCtopos}
	Let $f \colon \cC \to \cD$ be a morphism of $(\infty,1)$-toposes in $\presl$. Then $f$ is left exact if and only if the commutative diagram
	\[
		\begin{tikzcd}
			\cC \arrow[r,"f"] \arrow[d,"\eta_{\cC}"] & \cD \arrow[d,"\eta_{\cD}"] \\
			\Sym(\cC) \arrow[r,"\Sym(f)"] & \Sym(\cD)
		\end{tikzcd}
	\]
	is vertically right adjointable in $\presl$ (cf. \cref{prop:topos-unit-char}).
\end{lemma}
\begin{proof}
	Suppose that the square is adjointable, so that $f \omega_\cC \simeq \omega_\cD M(f)$. To see that $f$ is left exact, let $I$ be a finite category and consider the diagram
	\[\begin{tikzcd}[ampersand replacement=\&, column sep = {5em, between origins}]
		{f\omega_\cC\eta_\cC(\lim_I c_i)} \&\& {f(\lim_I \omega_\cC \eta_\cC (c_i))} \& \\
		\& {f \omega_\cC(\lim_I \eta_\cC (c_i))} \&\& {\lim_I f \omega_\cC \eta_\cC (c_i).}
		\arrow["\sim", from=1-1, to=1-3]
		\arrow[from=1-1, to=2-2]
		\arrow[from=1-3, to=2-4]
		\arrow[from=2-2, to=1-3]
		\arrow["\sim", from=2-2, to=2-4]
	\end{tikzcd}\]
	The top horizontal arrow is an equivalence since $\omega_\cC \eta_\cC \simeq \id_\cC$. The bottom horizontal map is an equivalence since $f \omega_\cC \simeq \omega_\cD M(f)$ and both $\omega_\cD$ and $M(f)$ preserve finite limits.
	We conclude by two-out-of-six that $f$ preserves finite limits.

	 We now prove the converse. Choose a regular cardinal $\kappa$ such that $\cC$, $\cD$, and $f$ are $\kappa$-accessible and $f$ preserves $\kappa$-compact objects. By \cite[Proposition 5.4.7.4]{HTT}, we may assume without loss of generality (by choosing $\kappa$ large enough) that the full subcategories on $\kappa$-compact objects $\cC_\kappa$ and $\cD_\kappa$ are closed under finite limits. It then follows as in the proof of Proposition 6.1.5.3 of op.\ cit.\ that the reflective localizations $\cP(\cC_\kappa) \to \cC$ and $\cP(\cD_\kappa) \to \cD$ are left exact. Borrowing the notation from \cref{prop:topos-unit-char}, we see that we have a commutative diagram
	\[\begin{tikzcd}
		{\cP(\cC_\kappa^{\lex})} & {\cP(\cC_\kappa)} & \cC \\
		{\cP(\cD_\kappa^{\lex})} & {\cP(\cD_\kappa)} & \cD{.}
		\arrow["{\rho_!}", from=1-1, to=1-2]
		\arrow["{(f^{\lex}_\kappa)_!}", from=1-1, to=2-1]
		\arrow[from=1-2, to=1-3]
		\arrow["{(f_{\kappa})_!}", from=1-2, to=2-2]
		\arrow["f", from=1-3, to=2-3]
		\arrow["{\rho_{!}}"', from=2-1, to=2-2]
		\arrow[from=2-2, to=2-3]
	\end{tikzcd}\]
where the left square commutes by \cref{exa:lex-cats} since $f_\kappa$ preserves finite limits.
Invoking the dual of \cref{BCeasy}, it suffices to produce an equivalence $f \circ \omega_{\cC} \simeq \omega_{\cD}\circ \Sym^{\kappa}(f_{\kappa})$. This may be checked after precomposing along the localization functors $L_{\Sigma_{\cC_{\kappa}}^{\kappa}} \colon \cP(\cC_\kappa^{\lex}) \to \Sym^{\kappa}(\cC_\kappa)$ and $L_{\Sigma_{\cD_{\kappa}}^{\kappa}} \colon \cP(\cD_\kappa^{\lex}) \to \Sym^{\kappa}(\cD_\kappa)$, and one readily verifies that the outer commutative square in the diagram above provides the desired identification.
\end{proof}

\begin{theorem}\label{thm:symtopos}
	There exists an oplax-idempotent monadic adjunction
	\[
		\Sym \colon \presl \llra \LTop \colon U,
	\]
	where $U$ denotes the forgetful functor. 
\end{theorem}
\begin{proof}
	The result now follows from the oplax variant of \cref{main-thm-inverse-2} together with \cref{prop:topos-unit-char} and \cref{lem:BCtopos}.
\end{proof}

\bibliographystyle{amsalpha}
\bibliography{references}

\end{document}